\documentclass[12pt,a4paper]{article}
\usepackage[cp1251]{inputenc}
\usepackage[english]{babel}
\usepackage{color}
\usepackage[colorlinks=true,linkcolor=blue,filecolor=blue,citecolor=blue,urlcolor=blue]{hyperref}
\usepackage{amsmath,amssymb,amsthm}
\usepackage{graphicx}
\usepackage{comment}
\usepackage{enumitem}
\usepackage{subfigure,multirow}
\usepackage{amsthm}
\usepackage{thmtools}
\declaretheoremstyle[notefont=\bfseries,notebraces={}{},%
   headpunct={},postheadspace=1em,bodyfont=\it]{mystyle}
\declaretheorem[style=mystyle,numbered=no,name=Theorem]{thm-hand}

\usepackage{geometry}
\DeclareMathOperator{\sgn}{sgn}

\DeclareMathOperator{\Kil}{Kil}

\DeclareMathOperator{\Lip}{Lip}
\DeclareMathOperator{\Image}{Im}
\DeclareMathOperator{\Real}{Re}

\DeclareMathOperator{\interior}{int}

\DeclareMathOperator{\ch}{cosh}
\DeclareMathOperator{\sh}{sinh}
\DeclareMathOperator{\arctg}{arctan}

\DeclareMathOperator{\ad}{\mathrm{ad}}
\DeclareMathOperator{\Ad}{\mathrm{Ad}}
\DeclareMathOperator{\Exp}{\mathrm{Exp}}
\DeclareMathOperator{\sspan}{\mathrm{span}}

\newcommand{\tmax}{t_{\mathrm{max}}}
\newcommand{\tconj}{t_{\mathrm{conj}}}
\newcommand{\tcut}{t_{\mathrm{cut}}}

\newcommand{\g}{\mathfrak{g}}

\newcommand{\R}{\mathbb{R}}
\newcommand{\C}{\mathbb{C}}

\newcommand{\Z}{\mathbb{Z}}

\newcommand{\SU}{\mathrm{SU}}
\newcommand{\SL}{\mathrm{SL}}

\newcommand{\SO}{\mathrm{SO}}

\newcommand{\const}{\mathrm{const}}

\newcommand{\hhh}{\bar{h}_3}
\newcommand{\st}{\sin{\tau}}
\newcommand{\ct}{\cos{\tau}}
\newcommand{\sht}{\sinh{\tau}}
\newcommand{\cht}{\cosh{\tau}}

\newcommand{\arglh}{\frac{t \mu h_3}{2 I_1}}

\newcommand{\A}{\mathcal{A}}

\newcommand{\Conj}{\mathrm{Conj}}
\newcommand{\Cut}{\mathrm{Cut}}

\DeclareMathOperator{\arcsinh}{arcsinh}

\theoremstyle{definition}
\newtheorem{definition}{Definition}
\newtheorem{remark}{Remark}

\theoremstyle{plain}
\newtheorem{corollary}{Corollary}
\newtheorem{lemma}{Lemma}
\newtheorem{theorem}{Theorem}
\newtheorem{proposition}{Proposition}

\newenvironment{enumerate*}%
  {\begin{enumerate}%
    \setlength{\itemsep}{1pt}%
    \setlength{\parskip}{1pt}}%
  {\end{enumerate}}

\title{One-parametric series of $\SO_{1,1}$-symmetric (sub-)Lorentzian structures on the universal covering of $\SL_2(\R)$\footnote{This work was supported by the Russian Science Foundation under grant no.~25-21-00681, \href{https://rscf.ru/project/25-21-00681/}{https://rscf.ru/project/25-21-00681/} and performed in Ailamazyan Program Systems Institute of Russian Academy of Sciences.}}

\author{
A.\,V.~Podobryaev \\ A.\,K.~Ailamazyan Program Systems
Institute of RAS \\ \tt{alex@alex.botik.ru} \\
}

\date{}

\begin{document}

\maketitle

\begin{abstract}
We consider a one-parametric series of left-invariant Lorentzian structures on the universal covering of the Lie group $\SL_2(\R)$.
These structures have $\SO_{1,1}$-symmetry and they are deformations of the anti-de\,Sitter Lorentzian manifold.
We study the global optimality of extremal trajectories, i.e., we describe the longest arcs.
The sub-Lorentzian structure appears as a limit case of the considered series of Lorentzian structures.
We study how the several properties of the Lorentzian structures deform to the properties of the sub-Lorentzian structure.


\end{abstract}

\section*{\label{sec-introduction}Introduction}

A Lorentzian structure on a manifold $M$ of dimension $\dim{M} = n+1$ is a non-degenerate quadratic form of signature $(1,n)$.
With the help of this quadratic form one can define the length of non-spacelike vectors and then the length of an admissible curve, i.e.,
such a curve that almost all its tangent vectors are non-spacelike. The longest arc from the point $x_0 \in M$ to the point $x_1 \in M$ is the curve where the supremum
of the Lorentzian length of the admissible curves from the point $x_0$ to the point $x_1$ is achieved if it exists.
Notice that the notion of the shortest arc is meaningless.

We consider a series of left-invariant Lorentzian structures on the universal covering of the Lie group $\SL_2(\R) \simeq \SU_{1,1}$.
If the Lorentzian structure is defined by the Killing form on this space, then this manifold is called \emph{the anti-de\,Sitter space}.
We consider a one-parametric deformation of the Lorentzian metric of the anti-de\,Sitter space.
So, the anti-de\,Sitter space is a particular case in our series of Lorentzian manifolds.
Moreover, the anti-de\,Sitter case splits this one-parametric series into the two subseries,
named oblate and prolate cases, respectively. The names of this subseries are due to the shape of the corresponding future cone.
The limit case of the oblate series is the sub-Lorentzian manifold. In this case, the future cone lies in some hyperplane
which means an additional non-holonomic constraint.

Moreover, the Lorentzian structures in our series have $\SO_{1,1}$-symmetry.
This allows us to write an explicit parametrization of extremal trajectories as products of two one-parametric subgroups.
In particular, these Lorentzian structures are geodesic orbit.
See~\cite{dusek-kowalski,chen-wolf-zhang} for some properties of geodesic orbit Lorentzian manifolds.

We use the language and the methods of geometric control theory~\cite{agrachev-sachkov} to investigate the attainable sets, the caustic, the cut locus and the domain where
the longest arcs exist. This approach was recently applied to several left-invariant Lorentzian structures: on the hyperbolic plane~\cite{sachkov-l1,sachkov-l2},
on the two-dimensional anti-de\,Sitter space~\cite{ali-sachkov}, on the Berger type Lorentzian manifolds~\cite{podobryaev-sl2-axisymm}
(in particular, on another series of Lorentzian structures on the universal covering of the Lie group $\SL_2(\R)$).

The notion of the sub-Lorentzian structure was introduced in works of M.~Grochowski~\cite{grochowski1,grochowski2}.
The sub-Lorentzian structure that appears as a limit case of our series of Lorentzian structures is a part of the known classification of left-invariant contact three dimensional sub-Lorentzian structures~\cite{classification}. This sub-Lorentzian structure was particulary studied by E.~Grong and A.~Vassil'ev~\cite{grong-vasiliev},
see also~\cite{chang-markina-vasilev}.
A lot of results about this limit structure follow from the general results about the whole subseries of oblate Lorentzian structures.

The main results of this paper are the following.

In the oblate case, the cut locus coincides with the first caustic and does not depend on the parameter.
The cut locus and the caustic for the limit sub-Lorentzian structure are the same.
However, the attainable set depends on the parameter and the attainable set for the limit structure does not coincide with the limit of the attainable sets in the series.
The reason is that the boundary of the sub-Lorentzian attainable set consists of abnormal extremal trajectories that are different from the light-like extremal trajectories in the Lorentzian case.
These abnormal extremal trajectories are concatenations of light-like arcs and have at most one switching of light-like control.

In the prolate case, the system is completely controllable, i.e., the attainable set coincides with the whole group.
There are admissible loops passing through every point. This implies that the longest arcs never exist.
However, we find some bounds of the first conjugate time and the first Maxwell time for the symmetries of the problem.
We prove that the first conjugate point appears earlier than two different extremal trajectories meet one another, i.e.,
earlier than a Maxwell point in this case.

This paper has the following structure.
In Section~\ref{sec-optctrlprob-statement} we state the problem of finding the Lorentzian longest arcs as an optimal control problem and
we describe a model of the manifold we are working on, i.e., the universal covering of the Lie group $\SL_2(\R)$.
Section~\ref{sec-geodesics} gives the equations of extremal trajectories in coordinates.
Next, we consider the oblate case (Section~\ref{sec-oblate}) and the prolate case (Section~\ref{sec-prolate}) separately.
These cases depend on the shape of the control set.

For the oblate case, first we describe the attainable set in Section~\ref{sec-oblate-atset}. Second, we find the first caustic in Section~\ref{sec-oblate-conj}.
Third, in Section~\ref{sec-oblate-opt} we study the optimality of extremal trajectories, we find the cut locus and the domain where the longest arcs exist.
Finally, we consider the limit case, i.e., the sub-Lorentzian structure in Section~\ref{sec-sub-Lorentzian-case}.

For the prolate case, we prove that the attainable set coincides with the whole manifold (Section~\ref{sec-prolate-atset}). It follows that there are no longest arcs.
However, in Sections~\ref{sec-prolate-conj}, \ref{sec-prolate-maxwell} we compare the first conjugate time and the first Maxwell time for the symmetries of the problem.

\section{\label{sec-optctrlprob-statement}Optimal control problem statement}

In this section, we define the universal covering of the Lie group we deal with and we give the multiplication law for this universal covering.
Next, we provide the formal statement of the optimal control problem that we consider and give some necessary notation.

Consider the Lie group
$$
\SU_{1,1} = \left\{
\left(
\begin{array}{cc}
z & w \\
\bar{w} & \bar{z} \\
\end{array}
\right) \, \Bigm| \, z,w \in \C
\right\} \simeq \SL_2(\R).
$$
We will denote by $G = \widetilde{\SU}_{1,1}$ its universal covering and by $\g$ the corresponding Lie algebra. As usual $\Ad$ and $\ad$ will be the adjoint representations of the Lie group and Lie algebra, respectively, and $\exp : \g \rightarrow G$ be the Lie exponential map. The universal covering is defined by the following formula:
$$
G \simeq \R \times \C \ni (c,w) \mapsto
\left(
\begin{array}{cc}
e^{ic}\sqrt{1 + |w|^2} & w \\
\bar{w} & e^{-ic}\sqrt{1 + |w|^2} \\
\end{array}
\right) \in \SU_{1,1}.
$$
It is not difficult to derive the multiplication rule in the group $G$. If $(c_1, w_1) \cdot (c_2, w_2) = (c, w)$, then
\begin{equation}
\label{eq-mult}
\begin{gathered}
c = c_1 + c_2 + \arctg{\frac{\Image{(w_1\bar{w}_2e^{-i(c_1+c_2)})}}{\sqrt{1+|w_1|^2}\sqrt{1+|w_2|^2} + \Real{(w_1\bar{w}_2e^{-i(c_1+c_2)})}}},\\
w = w_2\sqrt{1+|w_1|^2}e^{ic_1} + w_1\sqrt{1+|w_2|^2}e^{-ic_2}.\\
\end{gathered}
\end{equation}

Let $e_1,e_2,e_3$ be a basis of the Lie algebra $\g$ such that
$$
e_1 = \frac{1}{2}\left(
\begin{array}{rr}
i & 0\\
0 & -i\\
\end{array}
\right), \qquad
e_2 = \frac{1}{2}\left(
\begin{array}{rr}
0 & 1\\
1 & 0\\
\end{array}
\right), \qquad
e_3 = \frac{1}{2}\left(
\begin{array}{rr}
0 & i\\
-i & 0\\
\end{array}
\right),
$$
\begin{equation}
\label{eq-commutators}
[e_1, e_2] = e_3, \qquad [e_1,e_3] = -e_2, \qquad [e_2,e_3] = -e_1.
\end{equation}
Consider the following left-invariant optimal control problem on the group $G$ with the control set (which is a cone)
$$
C = \{u_1e_1 + u_2e_2 + u_3e_3 \in \g \, | \, I_1 u_1^2 - I_2 u_2^2 - I_3 u_3^2 \geqslant 0, \ u_1 > 0\}, \qquad I_1,I_2,I_3 > 0.
$$
The goal is to find a Lipschitz curve $x : [0,t_1] \rightarrow G$ and a control $u \in L^{\infty}([0,t_1],C)$ such that
\begin{equation}
\label{eq-optctrl-problem}
x(0) = e, \ x(t_1) = x_1, \qquad \dot{x}(t) = L_{x(t) *} u(t), \qquad \int\limits_0^{t_1}{\sqrt{I_1 u_1^2 - I_2 u_2^2 - I_3 u_3^2} \, dt} \rightarrow \max,
\end{equation}
where $e \in G$ is the identity element, $x_1 \in G$ and $t_1 > 0$ are fixed, while $L_x$ denotes the left-shift by an element $x \in G$ and $L_{x *}$ is its differential.

In this paper, we consider only the case when $I_1 = I_2$. We call it \emph{the $\SO_{1,1}$-symmetric case} since the Lorentzian structure is invariant under the action of the Lie group $\SO_{1,1} = \{\exp{(\ad{te_3})} \, | \, t \in \R\}$. This group acts on the Lie algebra $\g$ by hyperbolic rotations in the plane $\sspan{\{e_1,e_2\}}$.
Also this group
$$
\SO_{1,1} = \left\{
\left(
\begin{array}{cc}
\ch{t} & i\sh{t} \\
-i\sh{t} & \ch{t} \\
\end{array}
\right) \, \Bigm| \, t \in \R
\right\}
$$
acts on the Lie group $G$ by conjugation.
In other words, we consider an invariant Lorentzian structure on the universal covering of the homogeneous space $(\SU_{1,1} \times \SO_{1,1}) / \SO_{1,1}$, where the stabilizer $\SO_{1,1}$ is embedded into the group $\SU_{1,1} \times \SO_{1,1}$ in the anti-diagonal way.

Let us introduce the following parameter that measures the oblateness of the set of vectors with the unit Lorentzian length
$$
\mu = \frac{I_1}{I_3} - 1, \qquad \mu \in (-1, +\infty).
$$
We will use the following names for cases:
$$
\begin{array}{ll}
\mu < 0, \quad \text{i.e.}, \quad I_1 = I_2 < I_3 & \qquad \text{\emph{the oblate case}}, \\
\mu = 0, \quad \text{i.e.}, \quad I_1 = I_2 = I_3 & \qquad \text{\emph{the symmetric case}}, \\
\mu > 0, \quad \text{i.e.}, \quad I_1 = I_2 > I_3 & \qquad \text{\emph{the prolate case}}. \\
\end{array}
$$
If $\mu \rightarrow -1$, i.e., $I_3 \rightarrow +\infty$, then the problem statement~\eqref{eq-optctrl-problem} tends to the sub-Lorentzian one considered in~\cite{grong-vasiliev}. The symmetric case was studied also in~\cite{grong-vasiliev}, and more precisely in~\cite{podobryaev-sl2-axisymm}.
So, we do not consider the symmetric case in this paper.

\begin{remark}
\label{rem-graph}
We plot graphs of extremal trajectories in the $(c,w)$-space.
Note that we consider the $c$-axis as the vertical axis.
So, the cone of admissible velocities, i.e., the future cone, is directed upward.
Moreover, we plot graphs in the $\SO_{1,1}$-factor, i.e., in the plane $\Real{w} = 0$.
\end{remark}

\section{\label{sec-geodesics}Equations of extremal trajectories}

It follows from the Pontryagin maximum principle~\cite{pontryagin} (see its modern statement in~\cite[Th.~12.3]{agrachev-sachkov}) that any optimal trajectory with an optimal control $\hat{u} = (\hat{u}_1,\hat{u}_2,\hat{u}_3)$ is a part of \emph{an extremal trajectory} that is a projection of an extremal.
\emph{An extremal} is a trajectory $\lambda : [0,t_1] \rightarrow T^*G$ of a Hamiltonian vector field on the cotangent bundle $T^*G$
with a Hamiltonian $H_{\hat{u}}^{\nu}$ such that $H_{\hat{u}}^{\nu}(\lambda) = \max\limits_{u \in C}{H^{\nu}_u(\lambda)}$ and $(\hat{\lambda}, \nu) \neq 0$, where
$$
H_{u}^{\nu} = u_1h_1 + u_2h_2 + u_3h_3 + \frac{\nu}{2}\sqrt{I_1 u_1^2 - I_1 u_2^2 - I_3 u_3^2},
$$
and the functions $h_i(\lambda) = \langle L_{\pi(\lambda)}e_i, \lambda \rangle$, $\lambda \in T^*G$, $i=1,2,3$ are linear Hamiltonians on the fibers of the cotangent bundle,
where $\pi : T^*G \rightarrow G$ is the natural projection.

We may consider the functions $h_1,h_2,h_3$ as coordinates in the dual space of the Lie algebra $\g^* = T^*_eG$.
The extremals satisfy the following differential equations in the trivialization of the cotangent bundle $T^*G$ by left-shifts.

\begin{proposition}
\emph{(1)} The extremals are solutions of the following Hamiltonian system corresponding to the maximized Hamiltonian
$H = \frac{1}{2}\left( -\frac{h_1^2}{I_1} + \frac{h_2^2}{I_1} + \frac{h_3^2}{I_3} \right)$
\begin{equation}
\label{eq-hamiltonian-system}
\begin{array}{rcl}
\dot{h}_1(t) & = & \frac{\mu}{I_1} h_2(t) h_3(t), \\
\dot{h}_2(t) & = & \frac{\mu}{I_1} h_1(t) h_3(t), \\
\dot{h}_3(t) & = & 0, \\
\dot{x}(t) & = & L_{x(t)} \left( -\frac{h_1(t)}{I_1}e_1 + \frac{h_2(t)}{I_1}e_2 + \frac{h_3(t)}{I_3}e_3 \right). \\
\end{array}
\end{equation}
\emph{(2)} The initial covectors of normal extremals \emph{(}corresponding to $\nu = 1$\emph{)} are located on the half of the hyperboloid $H = -\frac{1}{2}$, $h_1 < 0$. These normal extremals are time-like and have unit Lorentzian velocity. \\
\emph{(3)} The initial covectors of abnormal extremals \emph{(}corresponding to $\nu = 0$\emph{)} are located on the half of the cone $H = 0$, $h_1 < 0$.
These abnormal extremals are light-like.
\end{proposition}

\begin{figure}
\centering{
\minipage{0.45\textwidth}
  \centering{\includegraphics[width=0.7\linewidth]{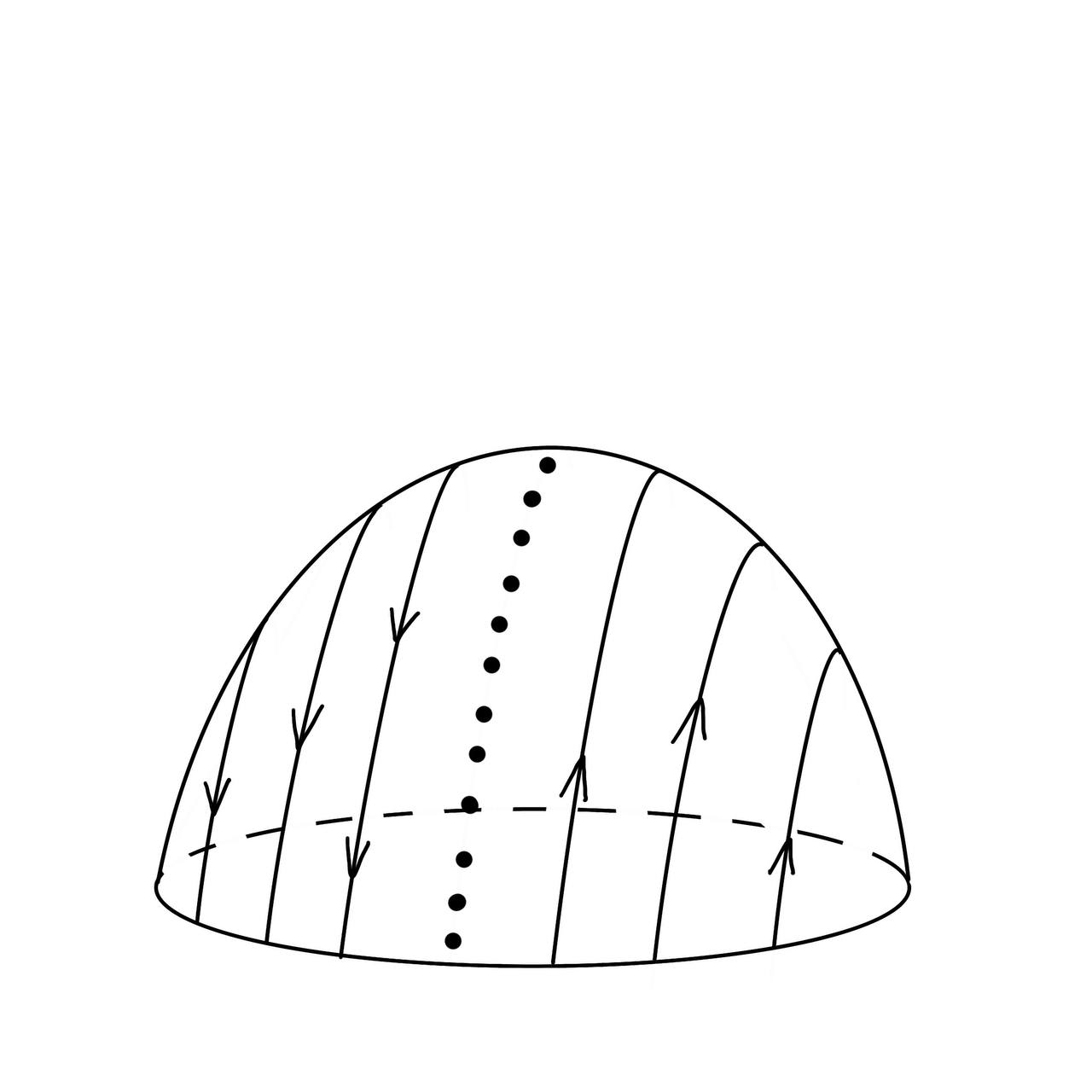} \\ $H(h) = -\frac{1}{2}$}
\endminipage
\hfil
\minipage{0.45\textwidth}
  \centering{\includegraphics[width=0.7\linewidth]{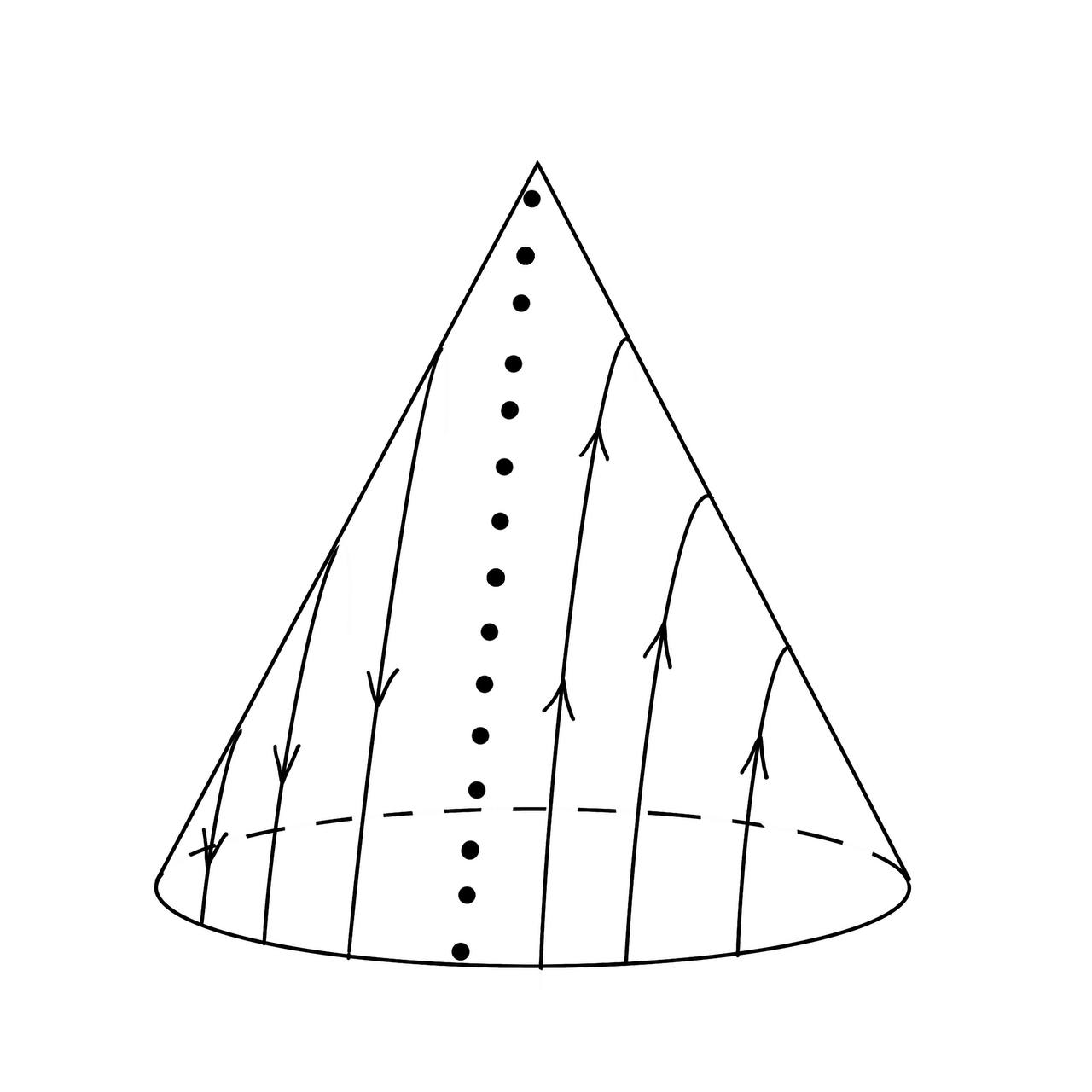} \\ $H(h) = 0$}
\endminipage}
\caption{\label{fig-vertsubsystem}The phase portraits for the Hamiltonian system for covectors.
The normal case corresponding to time-like extremal trajectories (left) and the abnormal case corresponding to light-like extremal trajectories (right).}
\end{figure}

\begin{proof}
It is well known that the Lorentzian length functional can be replaced by the quadratic energy functional.
Then, the maximum condition implies $u_1 = -\frac{h_1}{I_1}$, $u_2 = \frac{h_2}{I_1}$, $u_3 = \frac{h_3}{I_3}$.
Next, the Hamiltonian system for covectors in the form $\dot{h}_i = \{H, h_i\}$ (where $\{\,\cdot\, , \,\cdot\, \}$ is the Poisson bracket in the space $\g^*$)
can be derived using the equality $\{h_i,h_j\} = \langle [e_i,e_j], \,\cdot\, \rangle$ and the commutator relations~\eqref{eq-commutators}.
We refer to~\cite[Crl.~1, Rem.~8]{podobryaev-extr} for details.
\end{proof}

\begin{proposition}
\label{prop-geodesics}
An extremal trajectory with an initial covector $h(0) \in \g^*$ has the following parametrization
$$
g(t) = \exp{\frac{t}{I_1}\bigl(-h_1(0)e_1 + h_2(0)e_2 + h_3(0)e_3\bigr)} \cdot \exp{\left(\frac{t\mu h_3(0)}{I_1} e_3\right)}.
$$
\end{proposition}

\begin{proof}
Let us check by direct computation that the curve $\lambda(t) = L_{x(t)*}h(t)$ is an extremal, i.e., a solution of the Hamiltonian system, where
$h(t) = (\Ad{\exp{(\frac{t\mu h_3(0)}{I_1} e_3)}})^* h(0) = (\Ad^*{\exp{(-\frac{t\mu h_3(0)}{I_1} e_3)}}) h(0)$.

First, let us check the Hamiltonian subsystem for covectors. The definition of $h(t)$ means that the covector at a moment $t$ is the result of the hyperbolic rotation of the initial covector
$$
\begin{array}{ccl}
\left(
\begin{array}{l}
h_1(t) \\
h_2(t) \\
\end{array}
\right)
& = & R^h_{\frac{t\mu h_3(0)}{I_1}}
\left(
\begin{array}{l}
h_1(0) \\
h_2(0) \\
\end{array}
\right), \\
h_3(t) & = & \const, \\
\end{array} \qquad \text{where} \qquad
R^h_a = \left(
\begin{array}{rr}
\cosh{a} & \sinh{a}\\
\sinh{a} & \cosh{a}\\
\end{array}
\right).
$$
Obviously, the covector $h(t)$ satisfies equations~\eqref{eq-hamiltonian-system}, see also Fig.~\ref{fig-vertsubsystem}.

Second, denoting by $L_x$ and $R_x$ the left- and the right-shift by an element $x \in G$, we obtain
$$
\dot{x}(t) = L_{\exp{\frac{t}{I_1}\left(-h_1(0)e_1 + h_2(0)e_2 + h_3(0)e_3\right)} *} R_{\exp{\left(\frac{t\mu h_3(0)}{I_1} e_3\right)} *}
\frac{1}{I_1}\bigl(-h_1(0)e_1 + h_2(0)e_2 + h_3(0)e_3\bigr) +
$$
$$
+ L_{\exp{\frac{t}{I_1}\left(-h_1(0)e_1 + h_2(0)e_2 + h_3(0)e_3\right)} *} L_{\exp{\left(\frac{t\mu h_3(0)}{I_1}\right)} *} \frac{\mu h_3(0)}{I_1} e_3  =
$$
$$
= L_{x(t) *} \left( \left(\Ad{\exp{\left(-\frac{t\mu h_3(0)}{I_1} e_3\right)}}\right) \frac{1}{I_1}\bigl(-h_1(0)e_1 + h_2(0)e_2 + h_3(0)e_3\bigr) + \frac{\mu h_3(0)}{I_1} e_3 \right).
$$
Using the solution of the Hamiltonian subsystem for covectors, since
in the basis $-e_1,e_2,e_3$ the operator $\Ad{\exp{\left(-\frac{t\mu h_3(0)}{I_1} e_3\right)}}$ looks like the hyperbolic rotation by $\frac{t\mu h_3(0)}{I_1}$, we get
$$
\left(\Ad{\exp{\left(-\frac{t\mu h_3(0)}{I_1} e_3\right)}}\right) \frac{1}{I_1}\Bigl(-h_1(0)e_1 + h_2(0)e_2 + h_3(0)e_3\Bigr) + \frac{\mu h_3(0)}{I_1} e_3  =
$$
$$
= \frac{1}{I_1}\Bigl(-h_1(t)e_1 + h_2(t)e_2 + h_3(0)e_3\Bigr) + \frac{1}{I_1}\Bigl(\frac{I_1}{I_3}-1\Bigr)h_3(0)e_3 =
$$
$$
= -\frac{h_1(t)}{I_1}e_1 + \frac{h_2(t)}{I_1}e_2 + \frac{h_3(t)}{I_3}e_3 = u_1(t)e_1 + u_2(t)e_2 + u_3(t)e_3.
$$
Hence, $\dot{x}(t) = L_{x(t) *} (u_1(t)e_1+u_2(t)e_2+u_3(t)e_3)$.
\end{proof}

\begin{remark}
\label{rem-isometry-group}
It turns out that the extremal trajectories are the products of two one-parametric subgroups (Proposition~\ref{prop-geodesics}).
We gave an explicit proof of this fact, however, it follows from some general considerations.
Note that the group $\SO_{1,1}$ acts on $\g$ like $\Ad{\exp{(te_3)}}$ by automorphisms that are hyperbolic rotations in the plane $\sspan{\{e_1, e_2\}}$.
It follows that the dual action on $\g^*$ by $\Ad^*{\exp{(te_3)}}$ induces the action of the same group on the group $G$ by isometries, see~\cite[Th.~1]{podobryaev-symmetries}. Whence, the isometry group of the corresponding Lorentzian structure contains the group $G \times \SO_{1,1}$.
The corresponding isotropy subgroup is $K^h = \left\{ \left((0, i \sinh{a}), -a\right) \, | \, a \in \R \right\}$.
Since the vertical part of the Hamiltonian vector field is tangent to the orbits of the isotropy group,
then the extremal trajectory corresponding to the initial covector $h \in \g^*$ is an orbit of one-parametric subgroup of isometries
$$
g(t) = \exp{t\left(d_hH + Z\right)},
$$
where $Z$ is an element of the Lie algebra of the isotropy subgroup such that $h(t) = (\Ad^*{tZ})h(0)$, see~\cite[Lem.~3.4]{podobryaev-hg}.
Note that in this paper this fact is proved for sub-Riemannian structure but the proof for arbitrary Hamiltonian system is quite the same.

We have $d_hH = -\frac{h_1}{I_1}e_1 + \frac{h_2}{I_1}e_2 + \frac{h_3}{I_3}e_3$ and $Z = -\frac{\mu h_3}{I_1} e_3$.
Finally, computing the exponential for the sum of commuting elements of the Lie algebra, we get the product of exponentials as in Proposition~\ref{prop-geodesics}.
\end{remark}

\begin{remark}
\label{rem-go}
Also we can apply the geodesic lemma for Lorentzian manifolds~\cite[Lem.~2.1]{dusek-kowalski}.
It follows, that the $\SO_{1,1}$-symmetric left-invariant structures on the universal covering of the Lie group $\SL_2(\R)$ are geodesic orbit, i.e.,
any extremal trajectory is an orbit of a one-parametric group of isometries.
\end{remark}

Introduce the following notation:
\begin{equation}
\label{eq-notation}
\Kil{h} = -h_1^2 + h_2^2 + h_3^2, \qquad
|h| = \sqrt{|\Kil(h)|}, \qquad \bar{h}_i = \frac{h_i}{|h|}, \ i=1,2,3, \qquad
\tau = \frac{t|h|}{2I_1}.
\end{equation}
Now let us deduce equations of extremal trajectories in coordinates.
These equations can be of three different types depending on type of an initial covector $h$ with respect to the Killing form $\Kil$ transferred to $\g^*$.

\begin{proposition}
\label{prop-geodesics-in-coordinates}
An extremal trajectory with an initial covector $h \in \g^*$ has the following parametrization\emph{:}
$$
c = \arg(q_0 + iq_1), \qquad w = q_2 + iq_3,
$$
where $q_0, q_1, q_2, q_3$ are defined as follows.\\
\emph{(1)} If $\Kil{h} < 0$, then
\begin{equation}
\label{eq-hyperbolic-time-geodesic}
\begin{array}{ccl}
q_0^e(\tau) & = & \ct \ch{(\tau\mu\bar{h}_3)} + \hhh \st \sh{(\tau\mu\bar{h}_3)},\\
\left(
  \begin{array}{r}
     -q_1^e(\tau)\\
     q_2^e(\tau)\\
  \end{array}
\right) & = &  \sin{\tau} R^h_{\tau \mu \hhh}
\left(
  \begin{array}{l}
     \bar{h}_1\\
     \bar{h}_2\\
  \end{array}
\right),\\
q_3^e(\tau) & = & \ct \sh{(\tau\mu\bar{h}_3)} + \hhh \st \ch{(\tau\mu\bar{h}_3)}.\\
\end{array}
\end{equation}
\emph{(2)} If $\Kil{h} = 0$, then
\begin{equation}
\label{eq-hyperbolic-light-geodesic}
\begin{array}{ccl}
q_0^p(t) & = & \ch{\left(\frac{t\mu h_3}{2I_1}\right)} + \frac{t}{2 I_1} h_3 \sh{\left(\frac{t\mu h_3}{2I_1}\right)},\\
\left(
  \begin{array}{r}
     -q_1^p(t)\\
     q_2^p(t)\\
  \end{array}
\right) & = & \frac{t}{2 I_1} R^h_{\arglh}
\left(
  \begin{array}{l}
     h_1\\
     h_2\\
  \end{array}
\right),\\
q_3^p(t) & = & \sh{\left(\frac{t\mu h_3}{2I_1}\right)} + \frac{t}{2 I_1} h_3 \ch{\left(\frac{t\mu h_3}{2I_1}\right)}.\\
\end{array}
\end{equation}
\emph{(3)} If $\Kil{h} > 0$, then
\begin{equation}
\label{eq-hyperbolic-space-geodesic}
\begin{array}{ccl}
q_0^h(\tau) & = & \cht \ch{(\tau\mu\bar{h}_3)} + \hhh \sht \sh{(\tau\mu\bar{h}_3)},\\
\left(
  \begin{array}{r}
     -q_1^h(\tau )\\
     q_2^h(\tau )\\
  \end{array}
\right) & = & \sht R^h_{\tau \mu \hhh}
\left(
  \begin{array}{l}
     \bar{h}_1\\
     \bar{h}_2\\
  \end{array}
\right),\\
q_3^h(\tau) & = & \cht \sh{(\tau\mu\bar{h}_3)} + \hhh \sht \ch{(\tau\mu\bar{h}_3)}.\\
\end{array}
\end{equation}
\end{proposition}

\begin{proof}
The proof is a direct computation of the product of two exponentials from Proposition~\ref{prop-geodesics}, using the formula
$\exp(x) = (\arg(q_0 + iq_1), q_2+iq_3)$, where
\begin{equation}
\label{eq-group-exp}
\begin{array}{llll}
q_0 + i q_1 = \cos{(\frac{|x|}{2})} + i \bar{x}_1 \sin{(\frac{|x|}{2})}, & q_2 + i q_3 = \sin{(\frac{|x|}{2})} (\bar{x}_2  + i \bar{x}_3) & \text{for} & \Kil(x) < 0,\\
q_0 + i q_1 = 1 + i \frac{x_1}{2}, & q_2 + i q_3 = \frac{1}{2}(x_2  + i x_3) & \text{for} & \Kil(x) = 0,\\
q_0 + i q_1 = \cosh{(\frac{|x|}{2})} + i \bar{x}_1 \sinh{(\frac{|x|}{2})}, & q_2 + i q_3 =  \sinh{(\frac{|x|}{2})} (\bar{x}_2 + i \bar{x}_3) & \text{for} & \Kil(x) > 0,\\
\end{array}
\end{equation}
and $x = x_1e_1 + x_2e_2 + x_3e_3 \in \g$.
\end{proof}

\begin{lemma}
\label{lem-light}
For initial covectors of light-like extremal trajectories we have $\bar{h}_3 = \pm \frac{1}{\sqrt{|\mu|}}$ or $\Kil{(h)} = 0$ and $\bar{h}_3 = 0$.
\end{lemma}

\begin{proof}
It follows from~Proposition~\ref{prop-geodesics}\,(3) that for an initial covector of a light-like extremal trajectory
$$
\frac{h_1^2}{I_1} - \frac{h_2^2}{I_1} - \frac{h_3^2}{I_3} = 0 \qquad \Rightarrow \qquad
\bar{h}_1^2 - \bar{h}_2^2 - (\mu + 1)\bar{h}_3^2 = 0 \quad \text{for} \quad \Kil{(h)} \neq 0.
$$
Assume that $\Kil{(h)} \neq 0$.
Since $\bar{h}_1^2 - \bar{h}_2^2 - \bar{h}_3^2 = -1$ for $\mu < 0$ and $\bar{h}_1^2 - \bar{h}_2^2 - \bar{h}_3^2 = 1$ for $\mu > 0$, then for a light-like extremal trajectory $\bar{h}_3^2 = \frac{1}{|\mu|}$.
If $\Kil{(h)} = 0$, then $\bar{h}_3 = 0$.
\end{proof}

\begin{lemma}
\label{lem-h3}
\emph{(1)} If $\mu < 0$, then for extremal trajectories of type~\emph{\eqref{eq-hyperbolic-time-geodesic}} we have any $\bar{h}_3 \in \R$ and
for extremal trajectories of type~\emph{\eqref{eq-hyperbolic-space-geodesic}} we have $\bar{h}_3 \in (-\infty,-\frac{1}{\sqrt{-\mu}}) \cup (\frac{1}{\sqrt{-\mu}}, +\infty)$.\\
\emph{(2)} If $\mu > 0$, then almost all extremal trajectories have the type~\emph{\eqref{eq-hyperbolic-time-geodesic}} and
the corresponding $\bar{h}_3 \in [-\frac{1}{\sqrt{\mu}}, \frac{1}{\sqrt{\mu}}]$.
There are two extremal trajectories of type~\emph{\eqref{eq-hyperbolic-light-geodesic}} with $\bar{h}_3 = 0$.
\end{lemma}

\begin{figure}
\centering{
\minipage{0.45\textwidth}
  \centering{\includegraphics[width=0.7\linewidth]{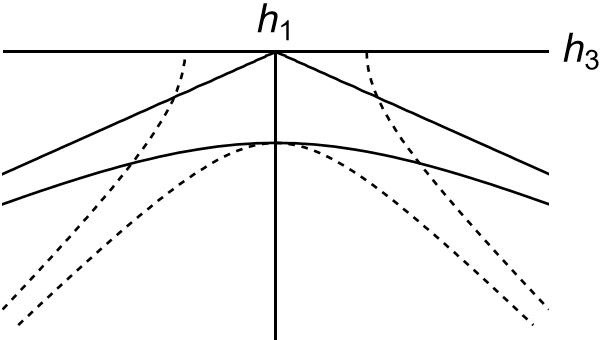} \\ $\mu < 0$}
\endminipage
\hfil
\minipage{0.45\textwidth}
  \centering{\includegraphics[width=0.7\linewidth]{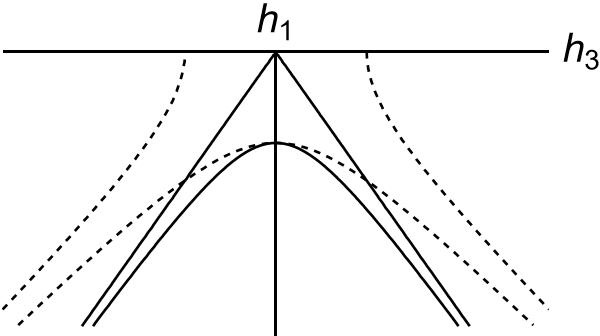} \\ $\mu > 0$}
\endminipage}
\caption{\label{fig-initial-covectors}Initial covectors for different types of extremal trajectories. The section $h_2 = 0$.
Solid lines are corresponding to the hyperboloid $H(h) = -\frac{1}{2}$ and the cone $H(h) = 0$.
Dashed lines are corresponding to the hyperboloids $|h| = 1$, or, equivalently, $\Kil{(h)} = \pm 1$.}
\end{figure}

\begin{proof}
Follows from Fig.~\ref{fig-initial-covectors} and Lemma~\ref{lem-light}.
\end{proof}

\section{\label{sec-oblate}The oblate case}

\subsection{\label{sec-oblate-atset}The attainable set}

\begin{proposition}
\label{prop-oblate-atset}
If $\mu < 0$, then the attainable set is
$$
\begin{array}{ll}
\A = & \left\{ (c,w) \in G \, \Bigm| \, |\tan{c}| \geqslant \sqrt{\frac{(\Real{w})^2 - \frac{\mu+1}{\mu}\sh^2{\tau}}{1 + (\Image{w})^2 + \frac{\mu+1}{\mu}\sh^2{\tau}}},
\ \text{where} \ \right.\\
& \left. -\ch{\tau} \sh{(\tau\sqrt{-\mu})} + \frac{1}{\sqrt{-\mu}} \sh{\tau} \ch{(\tau\sqrt{-\mu})} = |\Image{w}| \right\}. \\
\end{array}
$$
\end{proposition}

\begin{proof}
Let us describe the surface swept by the light-like extremal trajectories.
It follows from~Lemma~\ref{lem-light} that for almost all light-like extremal trajectories $\bar{h}_3 = \pm \frac{1}{\sqrt{-\mu}}$.
This implies that $\bar{h}_1^2 - \bar{h}_2^2 = -1 + \bar{h}_3^2 = \frac{\mu+1}{-\mu}$.
Hence, for coordinates of the light-like extremal trajectories with the initial covector $h$ from equations~\eqref{eq-hyperbolic-space-geodesic} we have
$$
q_2^2(\tau) - q_1^2(\tau) = \frac{\mu+1}{\mu}\sh^2{\tau}.
$$
Since $q_2(\tau) = \Real{w}$, $q_3(\tau) = \Image{w}$, we obtain
$$
\tan^2{c(\tau)} = \frac{q_1^2(\tau)}{q_0^2(\tau)} = \frac{q_2^2(\tau) - (q_2^2(\tau) - q_1^2(\tau))}{1 + q_3^2(\tau) + q_2^2(\tau) - q_1^2(\tau)} =
\frac{(\Real{w})^2 - \frac{\mu+1}{\mu}\sh^2{\tau}}{1 + (\Image{w})^2 + \frac{\mu+1}{\mu}\sh^2{\tau}}.
$$
It remains to determine the value $\tau$ depending on $w$ from the last equation of~\eqref{eq-hyperbolic-space-geodesic}
$$
q_3(\tau) = \ch{\tau} \sh{(\tau\mu\bar{h}_3)} + \bar{h}_3 \sh{\tau} \ch{(\tau\mu\bar{h}_3)} =
$$
$$
= \mp \ch{\tau} \sh{(\tau\sqrt{-\mu})} \pm \frac{1}{\sqrt{-\mu}} \sh{\tau} \ch{(\tau\sqrt{-\mu})} = \Image{w},
$$
since $\bar{h}_3 = \pm \frac{1}{\sqrt{-\mu}}$.
Let us show that for fixed $w$ there is the unique $\tau$.
Indeed, for  $\bar{h}_3 = \frac{1}{\sqrt{-\mu}}$ the function $q_3(\tau) = - \ch{\tau} \sh{(\tau\sqrt{-\mu})} + \frac{1}{\sqrt{-\mu}} \sh{\tau} \ch{(\tau\sqrt{-\mu})}$ increases, since for its derivative we have
$$
\begin{array}{ll}
q_3'(\tau) = & - \sh{\tau} \sh{(\tau\sqrt{-\mu})} - \sqrt{-\mu} \ch{\tau} \ch{(\tau\sqrt{-\mu})} + \\
 & + \frac{1}{\sqrt{-\mu}} \ch{\tau} \ch{(\tau\sqrt{-\mu})} + \sh{\tau} \sh{(\tau\sqrt{-\mu})} = \\
 & = \frac{\mu+1}{\sqrt{-\mu}} \ch{\tau} \ch{(\tau\sqrt{-\mu})} > 0, \\
\end{array}
$$
because $\mu > -1$. So, $\tau$ is the unique solution of the equation $q_3(\tau) = \Image{w}$.
Note that $q_1(\tau) > 0$. So, the sign of $\tan{c}$ depends on the sign of $q_0(\tau)$.

It follows from the geometrical statement of the Pontryagin maximum principle~\cite[Th.~12.1]{agrachev-sachkov} that
if a point belongs to the boundary of the attainable set, then there exists an abnormal extremal trajectory coming to this point.
Abnormal extremal trajectories are light-like extremal trajectories for the Lorentzian problem.

Let us prove that the light-like extremal trajectories sweep the boundary $\partial\A$ of the attainable set.
Notice that since $\mu < 0$ the control set $C$ is contained in the cone $C_0 = \{u_1e_1 + u_2e_2 + u_3e_3 \in \g \, | \, u_1^2 - u_2^2 - u_3^2 \leqslant 0\}$
that corresponds to the Killing form (or, equivalently, to the symmetric case $\mu = 0$).
This implies that the attainable set for the control set $C$ is contained in the attainable set for the control set $C_0$.
The last one is well known~\cite[Prop.~7\,(a)]{grong-vasiliev}
$$
\A_0 = \{ (c,w) \in G \, | \, c \geqslant \arctan{|w|} \}.
$$
In particular, this set is bounded from below with respect to the variable $c$, as well as our attainable set $\A \subset \A_0$.
Now assume by contradiction that there is a point $(c_0, w_0)$ on a light-like extremal trajectory that does not belong to $\partial\A$.
This means that the point $(c_0, w_0)$ lies in the interior of the vertical section of the attainable set $\A_{w_0} = \{(c,w) \in \A \, | \, w = w_0 \}$.
Since the set $\A_{w_0}$ is bounded from below, then $\inf{\A_{w_0}} > -\infty$.
Thus, $(\inf{\A_{w_0}}, w_0) \in \partial\A$. But there are no light-like trajectories coming to this point.
We get a contradiction.
\end{proof}

\subsection{\label{sec-oblate-conj}Conjugate points}

Recall some definitions.

\begin{definition}
\label{def-exp-map}
\emph{The exponential map} for an optimal control problem is the map
$$
\Exp : \g^* \times \R_+ \rightarrow G, \qquad \Exp(h, t) = \pi \circ e^{t\vec{H}}(h), \qquad h \in \g^*, \ t \in \R_+,
$$
where $e^{t\vec{H}}$ is the flow of the Hamiltonian vector field $\vec{H}$ corresponding to the Hamiltonian $H$ and
$\pi : T^*G \rightarrow G$ is the natural projection.
So, any extremal trajectory can be written in a form $\Gamma_h = \{\Exp(h,t) \, | \, t \in \R_+\}$ for some $h \in \g^*$ such that $h_1 < 0$ and $H(h) = -\frac{1}{2}$ or $H(h) = 0$.
\end{definition}

\begin{definition}
\label{def-conj}
A time $t \in \R_+$ is called \emph{a conjugate time along an extremal trajectory $\Gamma_h$ with an initial covector $h \in \g^*$} if the point $(h,t)$ is a critical point of the exponential map. The corresponding point $\Exp(h,t)$ is called \emph{a conjugate point}.
\end{definition}

We are interested in \emph{the first conjugate time} which we denote by $\tconj(h) \in \R_+ \cup \{+\infty\}$. The set of first conjugate points for all extremal trajectories we denote by $\Conj$ and call it \emph{the first caustic}.

\begin{theorem}
\label{th-oblate-conj}
If $\mu < 0$, then the first conjugate time equals
$$
\tconj(h) = \left\{
\begin{array}{lcl}
\frac{2\pi I_1}{|h|}, & \text{if} & \Kil{(h)} < 0,\\
+\infty, & \text{if} & \Kil{(h)} \geqslant 0.\\
\end{array}
\right.
$$
The first caustic is $\Conj = \{(\pi,w) \, | \, w \in i\R\}$.
\end{theorem}

\begin{proof}
We find the zeros of the Jacobian of the exponential map.
It is sufficient to consider an exponential map to the Lie group $\SU_{1,1}$ given by formulas~\eqref{eq-hyperbolic-time-geodesic}--\eqref{eq-hyperbolic-space-geodesic}.
Notice that the Jacobi matrix for the exponential map is the following product of two matrices
$$
\frac{\partial(q_0,q_1,q_2,q_3)}{\partial(\tau,\bar{h}_1,\bar{h}_2,\bar{h}_3)} \cdot \frac{\partial(\tau,\bar{h}_1,\bar{h}_2,\bar{h}_3)}{\partial(t,h_1,h_2,h_3)},
$$
where the second matrix is nondegenerate. So, it is sufficient to find the first zeros of the determinant of the first matrix.
To do this we will use formulas~\eqref{eq-hyperbolic-time-geodesic} and \eqref{eq-hyperbolic-space-geodesic} for $\Kil{(h)} \neq 0$.
The case $\Kil{(h)} = 0$ we investigate separately later.
Since $\frac{\partial q_1}{\partial \bar{h}_3} = \frac{\partial q_2}{\partial \bar{h}_3} = 0$, then
\begin{equation}
\label{eq-det}
\det{\frac{\partial(q_0,q_1,q_2,q_3)}{\partial(\tau,\bar{h}_1,\bar{h}_2,\bar{h}_3)}} =
\left(\frac{\partial q_1}{\partial \bar{h}_1}\frac{\partial q_2}{\partial \bar{h}_2} - \frac{\partial q_1}{\partial \bar{h}_2}\frac{\partial q_2}{\partial \bar{h}_1} \right) \cdot
\left(\frac{\partial q_0}{\partial \tau}\frac{\partial q_3}{\partial \bar{h}_3} - \frac{\partial q_3}{\partial \tau}\frac{\partial q_0}{\partial \bar{h}_3} \right).
\end{equation}
To make computations unique for both cases $\Kil{(h)} < 0$ and $\Kil{(h)} > 0$ let us define the following functions:
$$
c(\tau,h) = \left\{
\begin{array}{lcl}
\cos{\tau}, & \text{if} & \Kil{(h)} < 0,\\
\cosh{\tau}, & \text{if} & \Kil{(h)} > 0,\\
\end{array}
\right. \qquad
s(\tau,h) = \left\{
\begin{array}{lcl}
\sin{\tau}, & \text{if} & \Kil{(h)} < 0,\\
\sinh{\tau}, & \text{if} & \Kil{(h)} > 0,\\
\end{array}
\right. \qquad
\delta(h) = \sgn{(\Kil{(h)})}.
$$
Obviously, the following equalities are satisfied
$$
c^2(\tau,h) - \delta(h)s^2(\tau,h) = 1, \qquad c'(\tau,h) = \delta(h)s(\tau,h), \qquad s'(\tau,h) = c(\tau,h).
$$
First, let us find the first positive root of the first multiplier in formula~\eqref{eq-det}. This multiplier equals $s(\tau,h)$.
Hence, its first positive root is equal to $\pi$ when $\Kil{(h)} < 0$ and there are no roots when $\Kil{(h)} > 0$.

Second, let us prove that the first positive root of the second multiplier in formula~\eqref{eq-det} is equal to the first positive root of the first multiplier.
We have the following partial derivatives
\begin{equation}
\label{eq-derivatives}
\begin{array}{rcl}
\frac{\partial q_0}{\partial \tau} & = & (\delta(h) + \mu\bar{h}_3^2)s(\tau,h)\cosh{(\tau\mu\bar{h}_3)} + \bar{h}_3(1+\mu)c(\tau,h)\sinh{(\tau\mu\bar{h}_3)},\\
\frac{\partial q_3}{\partial \tau} & = & (\delta(h) + \mu\bar{h}_3^2)s(\tau,h)\sinh{(\tau\mu\bar{h}_3)} + \bar{h}_3(1+\mu)c(\tau,h)\cosh{(\tau\mu\bar{h}_3)},\\
\frac{\partial q_0}{\partial \bar{h}_3} & = & \tau\mu c(\tau,h)\sinh{(\tau\mu\bar{h}_3)} + s(\tau,h)\sinh{(\tau\mu\bar{h}_3)} + \tau\mu\bar{h}_3 s(\tau,h)\cosh{(\tau\mu\bar{h}_3)},\\
\frac{\partial q_3}{\partial \bar{h}_3} & = & \tau\mu c(\tau,h)\cosh{(\tau\mu\bar{h}_3)} + s(\tau,h)\cosh{(\tau\mu\bar{h}_3)} + \tau\mu\bar{h}_3 s(\tau,h)\sinh{(\tau\mu\bar{h}_3)}.\\
\end{array}
\end{equation}
Whence, the second multiplier in formula~\eqref{eq-det} equals
\begin{equation}
\label{eq-conj-full}
s(\tau,h) \left[ \tau\mu(\delta(h) - \bar{h}_3^2)c(\tau,h) + (\delta(h) + \mu\bar{h}_3^2) s(\tau,h) \right].
\end{equation}
It is sufficient to find the first positive root of expression in the square brackets.
Note that $\delta(h) + \mu\bar{h}_3^2 < 0$. Indeed, if $\delta(h) = -1$, then this expression is negative, since $\mu < 0$.
If $\delta(h) = 1$, then $1 + \mu\bar{h}_3^2 < 0$, since $\bar{h}_3^2 > -\frac{1}{\mu}$ in this case
by Lemma~\ref{lem-h3}\,(1).
Moreover, if $\tau$ is a root, then $c(\tau,h) \neq 0$, since otherwise $s(\tau,h) = 0$ and we get a contradiction.
So, $\tau$ is a root of the expression in square brackets if and only if
\begin{equation}
\label{eq-tgth}
\frac{s(\tau,h)}{c(\tau,h)} = -\tau\mu\frac{\delta(h) - \bar{h}_3^2}{\delta(h) + \mu\bar{h}_3^2}.
\end{equation}
Let us determine where the graph of the function $\frac{s(\tau,h)}{c(\tau,h)}$ intersects the line corresponding to the right-hand side of the equation above.
The answer is at Fig.~\ref{fig-conj}.

Consider the case $\delta(h) = -1$. We have
\begin{equation}
\label{eq-t-conj-coeff0}
-1 - \bar{h}_3^2 < 0, \quad -1 + \mu\bar{h}_3^2 < 0, \qquad \mu < 0 \quad \Rightarrow \quad -\mu\frac{-1 - \bar{h}_3^2}{-1 + \mu\bar{h}_3^2} > 0.
\end{equation}
Hence, the coefficient of the line is positive. Moreover, this coefficient is less than the derivative of $\tan$ at zero.
Indeed,
\begin{equation}
\label{eq-t-conj-coeff1}
-\mu \frac{-1 - \bar{h}_3^2}{-1 + \mu\bar{h}_3^2} < 1 \quad \Leftrightarrow \quad \mu + \mu\bar{h}_3^2 > -1 + \mu\bar{h}_3^2 \quad \Leftrightarrow \quad \mu > -1.
\end{equation}
It follows that for the first positive root we obtain $\tau \in (\pi, \frac{3\pi}{2})$.

Consider now the case $\delta(h) = 1$. Since $\mu < 0$ and $\bar{h}_3^2 > -\frac{1}{\mu} > 1$ by Lemma~\ref{lem-h3}\,(1), we have
\begin{equation}
\label{eq-s-conj-coeff0}
1 - \bar{h}_3^2 < 0, \quad 1 + \mu\bar{h}_3^2 < 0, \qquad \mu < 0 \quad \Rightarrow \quad -\mu\frac{1 - \bar{h}_3^2}{1 + \mu\bar{h}_3^2} > 0.
\end{equation}
So, the coefficient of the line is positive, but it is greater than the derivative of $\tanh$ at zero, since
\begin{equation}
\label{eq-s-conj-coeff1}
-\mu \frac{1 - \bar{h}_3^2}{1 + \mu\bar{h}_3^2} > 1 \quad \Leftrightarrow \quad -\mu + \mu\bar{h}_3^2 < 1 + \mu\bar{h}_3^2 \quad \Leftrightarrow \quad -\mu < 1.
\end{equation}
It follows that there are no positive roots.

\begin{figure}
\centering{
\minipage{0.45\textwidth}
  \centering{\includegraphics[width=0.7\linewidth]{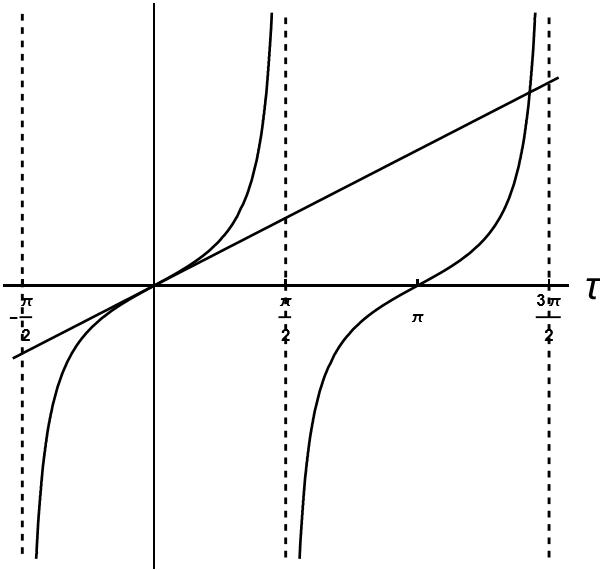} \\ $\Kil{(h)} < 0$}
\endminipage
\hfil
\minipage{0.45\textwidth}
  \centering{\includegraphics[width=0.7\linewidth]{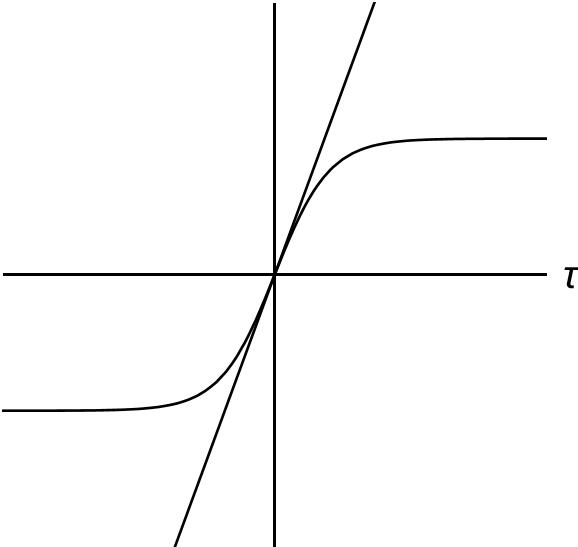} \\ $\Kil{(h)} > 0$}
\endminipage}
\caption{\label{fig-conj}The location of the first positive root of the equation~\eqref{eq-tgth}.}
\end{figure}
In any case when $\Kil{(h)} \neq 0$ the fist positive root of the Jacobian equals to the first positive foot of the equation $s(\tau,h) = 0$.
So, we get the required expression for the first conjugate time. Indeed, if $\tau = \pi$, then $t = \frac{2\pi I_1}{|h|}$, see notation~\eqref{eq-notation}.

To complete the calculation of the first conjugate time it is sufficient to prove that for initial covectors such that $\Kil{(h)} = 0$ the corresponding extremal trajectories have no conjugate points. Assume by contradiction that for such a covector $h$ the conjugate time is finite $\tconj(h) < +\infty$.
Since the conjugate points are isolated~\cite{agrachev} there exists $t_1 > \tconj(h)$
such that $t_1$ is not a conjugate time for the corresponding extremal trajectory $\Gamma_h$.
Consider a homotopy of this initial covector, i.e., a continuous curve $\gamma : [0,1] \rightarrow \g^*$ such that $\gamma(0) = h$, $\Kil{(\gamma(s))} < 0$ for $s > 0$
and $|\gamma(s)| < \frac{2\pi I_1}{t_1}$.
The last condition means that the are no conjugate points on the arc $\{\Gamma_{\gamma(s)}(t) \, | \, t \in [0,t_1]\}$ for $s > 0$.
The number of conjugate points on an arc of an extremal trajectory is equal to 
the Maslov index~\cite{arnold-index-maslova} of the path $\Lambda^s(t) = e^{-t\vec{H}}_* T^*_{\Gamma_{\gamma(s)}(t)}G$ 
in the Grassmanian of the Lagrange subspaces of the space $T_{(e,0)}T^*G$.
Thanks to homotopy invariance of the Maslov index~\cite{agrachev} the numbers of conjugate points on the arcs $\{\Gamma_{\gamma(s)}(t) \, | \, t \in [0,t_1]\}$ are equal for $s \in [0,1]$. Since there are no conjugate points for $s = 1$, it follows that there are no conjugate points for $s = 0$ as well.

Now let us determine the first caustic. We have
$$
\Conj = \left\{\Exp(h,\tconj(h)) \, \Bigm| \, \frac{h_1^2}{I_1} - \frac{h_2^2}{I_1} - \frac{h_3^2}{I_3} = 1, \ h_1 < 0, \ \Kil{(h)} < 0 \right\}.
$$
From formulas~\eqref{eq-hyperbolic-time-geodesic} we get
$$
q_0(\pi) = -\cosh{(\pi\mu\bar{h}_3)}, \quad q_1(\pi) = q_2(\pi) = 0, \quad q_3(\pi) = -\sinh{(\tau\mu\bar{h}_3)} \quad \text{for} \quad \bar{h}_3 \in \R.
$$
Whence, $c(\pi) = \arg{\{q_0(\pi) + iq_1(\pi)\}} = \pi$ and $w(\pi) \in i\R$.
\end{proof}

\subsection{\label{sec-oblate-opt}Optimality of extremal trajectories}

Let us remind two definitions.

\begin{definition}
\label{def-cut}
A time $\tcut(h) \in \R \cup \{+\infty\}$ is called \emph{the cut time for the extremal trajectory with the initial covector $h \in \g^*$} if the arc
$\{\Gamma_h(t) \, | \, t \in [0,t_1]\}$ is optimal for $t_1 \leqslant \tcut(h)$ and is not optimal for $t_1 > \tcut(h)$. The corresponding point $\Gamma_h(\tcut(h))$ is called \emph{a cut point}. The set of the cut points for all extremal trajectories is called \emph{the cut locus}, it is denoted by $\Cut$.
\end{definition}

\begin{definition}
\label{def-maxwell}
A point $x$ on an extremal trajectory $\Gamma_{h_1}$ is called \emph{a Maxwell point} if there exists a distinct extremal trajectory $\Gamma_{h_2}$ coming to the point $x$ at the same time with the same Lorentzian length $x = \Gamma_{h_1}(\tmax) = \Gamma_{h_2}(\tmax)$. The corresponding time $\tmax$ is called \emph{a Maxwell time}.
\end{definition}

It is well known that after a Maxwell point an extremal trajectory can not be optimal~\cite[Prop.~2.1]{sachkov-didona}.
Thus, we are interested in the first Maxwell time.
The natural reason for Maxwell time appearance are symmetries.
We find the first Maxwell time for symmetries as a function $\tmax: \g^* \rightarrow \R_+ \cup \{+\infty\}$ and prove that it is equal to the cut time.
More precisely, we prove that the exponential map is a diffeomorphism between the domain in $\g^* \times \R_+$ bounded by the first Maxwell time for symmetries and the domain $\interior{\A_{\tmax}}$, where $\interior{\A_{\tmax}}$ is interior of the attainable set by time less or equal to the first Maxwell time for symmetries.

\begin{definition}
\label{def-sym}
A pair of diffeomorphisms $(s, S)$
$$
s: \g^* \rightarrow \g^*, \qquad S : G \rightarrow G
$$
is called \emph{a symmetry of the exponential map} if $S \circ \Exp(h,t) = \Exp (s(h),t)$ for any $h \in \g^*$ and $t \in \R_+$.
\end{definition}

\begin{proposition}
\label{prop-sym}
The group of symmetries of the exponential map contains the group $\mathrm{Sym} = \mathrm{O}_{1,1} \times \Z_2$.
\end{proposition}

\begin{proof}
Any symmetry $s$ of the Hamiltonian system for covectors such that the dual map $s^*$ is an (anti-) automorphism of the Lie algebra $\g$ induces a symmetry of the exponential map $(s,S)$ such that $d_eS = s^*$, see~\cite[Th.~1]{podobryaev-symmetries}. Such symmetries are the hyperbolic rotations in the plane $(h_1,h_2)$, the reflection with respect to the plane $h_2 = 0$ and the reflection with respect to the plane $h_3 = 0$.

Obviously, the corresponding symmetries of the Lie group $G$ for reflections are also reflections with respect to the planes $\Real{w} = 0$ and $\Image{w} = 0$ in the model of the group $G$ with $(c,w)$-coordinates.

The hyperbolic rotation acts on the Lie group $\SU_{1,1}$ by hyperbolic rotation in the $(q_1,q_2)$-plane.
Hence, the fixed point of these rotations has the form $q_0 \in \R$, $q_1 = q_2 = 0$, $q_3 \in \R$, $q_0^2 - q_3^2 = 1$.
The set of fixed points in the universal covering $G$ is the union of lines $\bigcup\limits_{k \in \Z} \{(c,w) \, | \, c =\pi k, \, w \in i\R \}$.
The orbits of this actions on the universal covering $G$ are represented at Fig.~\ref{fig-orbits}.
\end{proof}

\begin{figure}
\centering{
\minipage{0.25\textwidth}
  \centering{\includegraphics[width=\linewidth]{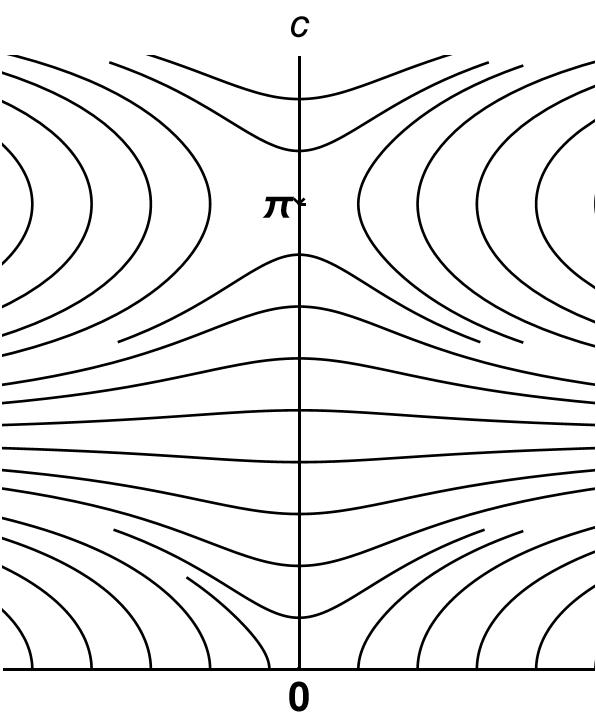} \\ $\Image{w} = 0$}
\endminipage
\hfil
\minipage{0.25\textwidth}
  \centering{\includegraphics[width=\linewidth]{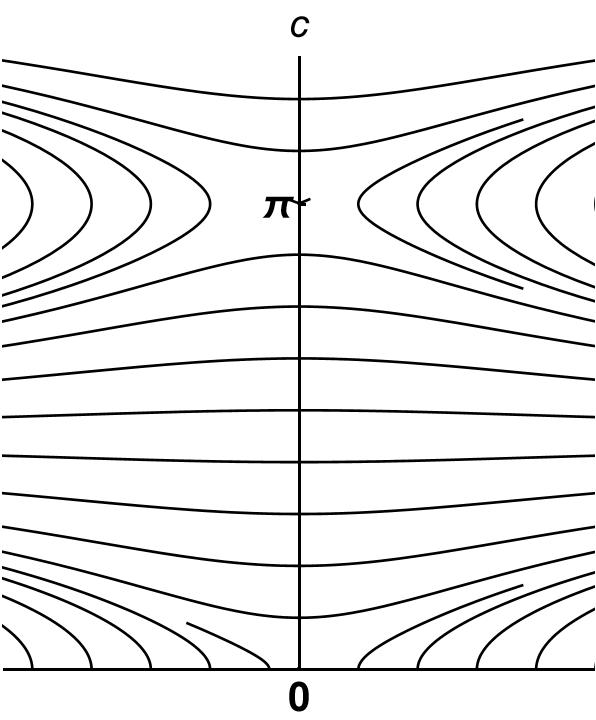} \\ $\Image{w} = \pm 2$}
\endminipage
\hfil
\minipage{0.25\textwidth}
  \centering{\includegraphics[width=\linewidth]{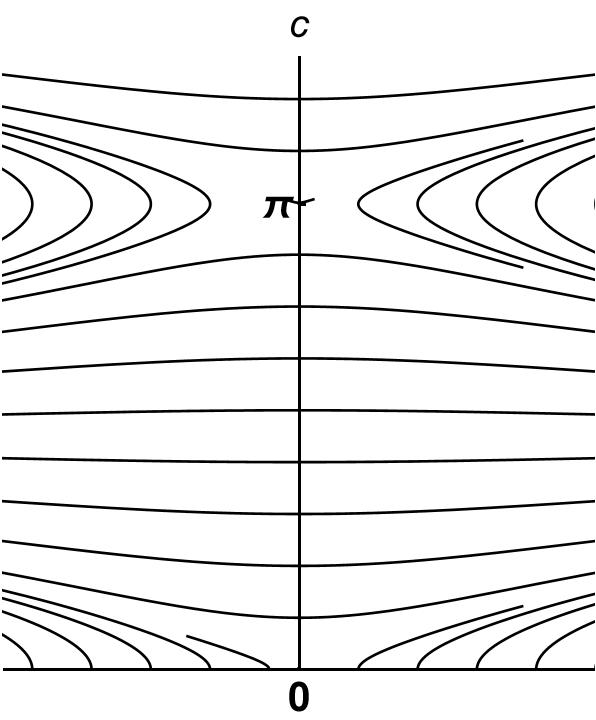} \\ $\Image{w} = \pm 3$}
\endminipage}
\caption{\label{fig-orbits}Orbits of the $\SO_{1,1}$-action at the group $G$ for different sections $\Image{w} = \const$.}
\end{figure}

\begin{proposition}
\label{prop-oblate-tmax}
If $\mu < 0$, then the first Maxwell time for symmetries equals
$$
\tmax(h) = \left\{
\begin{array}{rcl}
\frac{2\pi I_1}{|h|}, & \text{if} & \Kil{(h)} < 0,\\
+\infty, & \text{if} & \Kil{(h)} \geqslant 0.\\
\end{array}
\right.
$$
The corresponding set of Maxwell points is $\mathcal{M} = \{(\pi,w) \in G \, | \, w \in i\R \}$.
\end{proposition}

\begin{proof}
Let us define the Maxwell time corresponding to the hyperbolic rotations by $t_r$.
It is easy to see from Proposition~\ref{prop-geodesics-in-coordinates} that $t_r(h)$ equals $+\infty$ for initial covectors $h$ such that $\Kil{(h)} \geqslant 0$ and $t_r(h) = \frac{2\tau_r(h)I_1}{|h|}$ otherwise, where $\tau_r(h)$ is the first positive number such that $q_1(\tau_r) = q_2(\tau_r) = 0$.
Obviously, $\tau_r = \pi$.

Let us denote by $t_3$ the Maxwell time corresponding to the reflection with respect to the plane $h_3 = 0$.
Thus, for $\Kil{(h)} \neq 0$ the corresponding value $\tau_3(h)$ is the first positive root of the equation $q_3(t) = 0$.
Since $\cosh{(\tau_3(h)\mu\bar{h}_3)} \neq 0$, we get that for $\Kil{(h)} < 0$ the equality $\cos{\tau_3(h)} = 0$ implies $\sin{\tau_3(h)} = 0$, this is a contradiction.
So, since $\cos{\tau_3(h)} \neq 0$ the condition $q_3(\tau) = 0$ is equivalent to the condition
\begin{equation}
\label{eq-oblate-maxwell}
\begin{array}{lcl}
\tanh{(\tau\mu\bar{h}_3)} = -\bar{h}_3 \tan{\tau}, & \text{for} & \Kil{(h)} < 0,\\
\tanh{(\tau\mu\bar{h}_3)} = -\bar{h}_3 \tanh{\tau}, & \text{for} & \Kil{(h)} > 0.\\
\end{array}
\end{equation}
The similar considerations for $\Kil{(h)} = 0$ give the condition
$$
\tanh{\left(\frac{t\mu h_3}{2I_1}\right)} = -\frac{th_3}{2I_1}.
$$
See Fig.~\ref{fig-t3} for the plots of these functions.
Assume that $h_3 > 0$ (the case $h_3 < 0$ is similar, and if $h_3 = 0$ the corresponding extremal trajectory lies in the plane $\Image{w} = q_3 = 0$).
Note that the derivative at zero of the left-hand function is greater than the derivative of the right-hand function, i.e.,
$\mu\bar{h}_3 > -\bar{h}_3$, since $\mu > -1$ (or $\frac{\mu h_3}{2I_1} > -\frac{h_3}{2I_1}$). 
This means that for $\Kil{(h)} < 0$ we obtain $\tau_3(h) \in (\pi, \frac{3\pi}{2})$, since $\mu < 0$. It follows that in this case $t_3(h) > t_r(h)$. 
If $\Kil{(h)} \geqslant 0$, then $t_3(h) = +\infty$.

\begin{figure}
\centering{
\minipage{0.25\textwidth}
  \centering{\includegraphics[width=\linewidth]{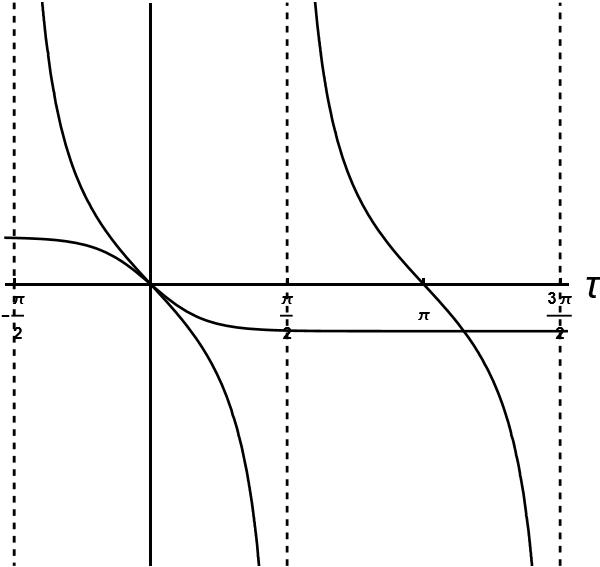} \\ $\Kil{(h)} < 0$}
\endminipage
\hfil
\minipage{0.25\textwidth}
  \centering{\includegraphics[width=\linewidth]{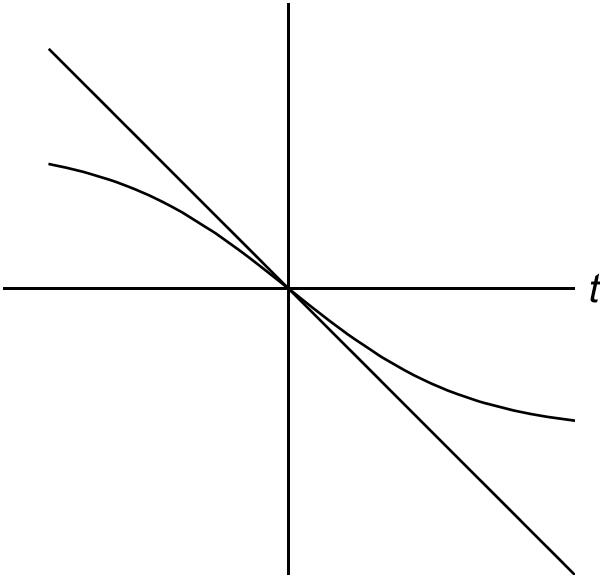} \\ $\Kil{(h)} = 0$}
\endminipage
\hfil
\minipage{0.25\textwidth}
  \centering{\includegraphics[width=\linewidth]{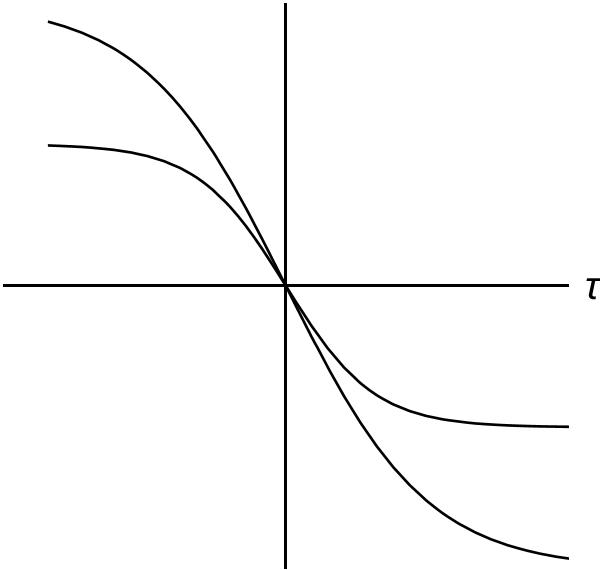} \\ $\Kil{(h)} > 0$}
\endminipage}
\caption{\label{fig-t3}The Maxwell time corresponding to the reflection with respect to the plane $h_3 = 0$ is greater or equal to the Maxwell time corresponding to the hyperbolic rotations.}
\end{figure}

Finally, note that the Maxwell time corresponding to the reflection with respect to the plane $h_2 = 0$ equals to the Maxwell time corresponding to the hyperbolic rotations.

It follows, that the first Maxwell set for symmetries is $\mathcal{M} = \{(\pi,w) \, | \, w \in i\R \}$.
\end{proof}

\begin{proposition}
\label{prop-oblate-Atmax}
The set that is attainable by extremal trajectories by time not exceeding the first Maxwell time is
$$
\A_{\tmax} = \{ \Exp(h,t) \, | \, 0 \leqslant t \leqslant \tmax(h) \} =
\left( \A \cap \left\{ (c,w) \Bigm| c < \pi - \arcsin{\frac{|\Real{w}|}{\sqrt{1+|w|^2}}} \right\} \right) \cup \mathcal{M}.
$$
\end{proposition}

\begin{proof}
The lower boundary of this set is the surface swept by the light-like extremal trajectories as we proved in Proposition~\ref{prop-oblate-atset}.
The upper boundary is the surface swept by the asymptotes of the $\SO_{1,1}$-orbits, see Fig.~\ref{fig-orbits}.
Let us prove this fact.

First, for any $w \in i\R$ and any $\varepsilon > 0$ there exists a time-like extremal trajectory of the type~\eqref{eq-hyperbolic-time-geodesic}
coming to the point $(c, w)$, where $\pi - \varepsilon < c < \pi$.
Indeed, otherwise there exists $w \in i\R$ such that for time-like extremal trajectories of the type~\eqref{eq-hyperbolic-time-geodesic}
coming to the Maxwell set to different sides of the point $(\pi, w)$
the partial derivatives $\frac{\partial}{\partial \tau}|_{\tau = \pi} q_3$ have different signs.
But from~\eqref{eq-derivatives} we know that
$$
\frac{\partial q_3}{\partial \tau}(\pi) = -\bar{h_3}(1 + \mu)\cosh{(\pi\mu\bar{h}_3)} = -\sgn{\bar{h}_3}.
$$
Thus, $\bar{h}_3 = 0$ and $w = 0$. But in this case by~\eqref{eq-hyperbolic-time-geodesic} all time-like extremal trajectories lie in the plane $q_3 = 0$ and the required condition is satisfied.

Second, if for any point $(\pi, w) \in \mathcal{M}$ there exists an infinitely close point $(c, w) \in \A_{\tmax}$ such that $c < \pi$ and $w \in i\R$, then
due to the presence of the $\SO_{1,1}$-symmetry the set $\A_{\tmax}$ contains the whole $\SO_{1,1}$-orbit of this point.
Hence, the set $\A_{\tmax}$ contains any orbit infinitely close to the asymptotes.

Third, if some time-like extremal trajectory of the type~\eqref{eq-hyperbolic-time-geodesic} comes to a point that is upper than the surface of these asymptotes,
then the same argument implies that there exists a one-parameter family of time-like extremal trajectories such that there is an extremal trajectory with $c$-coordinate greater than $\pi$ for the time less than the first Maxwell time. It follows that $q_1$ changes the sign. Whence, at some time moment less than the first Maxwell time we obtain $q_1 = q_2 = 0$.
This contradicts with the definition of the first Maxwell time.

It remains to write an equation for the upper boundary of the set $\A_{\tmax}$.
In the Lie group $\SU_{1,1}$ in $(q_0,q_1,q_2,q_3)$-coordinates we have the equation $q_1 = \pm q_2$ for the asymptotic $\SO_{1,1}$-orbits.
If we lift this surface to the universal covering $G$ at a point we get $q_1 = \sin{c}\sqrt{1+|w|^2}$.
Hence, $c = \pi - \arcsin{\frac{|q_2|}{\sqrt{1+|w|^2}}}$.
\end{proof}

\begin{theorem}
\label{th-oblate-cut}
If $\mu < 0$, then the cut time is equal to
$$
\tcut(h) = \left\{
\begin{array}{rcl}
\frac{2\pi I_1}{|h|}, & \text{if} & \Kil{(h)} < 0,\\
+\infty, & \text{if} & \Kil{(h)} \geqslant 0.\\
\end{array}
\right.
$$
The cut locus is $\Cut = \{(\pi,w) \in G \, | \, w \in i\R \}$.
\end{theorem}

\begin{proof}
First, note that the light-like extremal trajectories are optimal to the infinity.
Indeed, if we assume by contradiction that there is an admissible curve coming to some point of a light-like extremal trajectory with non-zero Lorentzian length,
then since this curve comes to the boundary of the attainable set $\A$ (see~Propositions~\ref{prop-oblate-atset}),
due to the Pontryagin maximum principle in geometric form~\cite[Th.~12.1]{agrachev-sachkov} this curve must be a light-like extremal trajectory and we get a contradiction.

Consider the domain in the pre-image of the exponential map bounded by the first Maxwell time for symmetries
$$
\mathcal{U} = \left\{ (h,t) \in \g^* \times \R_+ \, \Bigm| \, H(h) = -\frac{1}{2}, \ 0 < t < \tmax(h) \right\}.
$$
It is sufficient to prove that the exponential map $\Exp : \mathcal{U} \rightarrow \interior{\A_{\tmax}}= \A_{\tmax} \setminus \partial\A_{\tmax}$ is a diffeomorphism,
where $\partial\A_{\tmax} = \partial\A \cup \mathcal{M}$, and $\A_{\tmax}$ is the set attainable by extremal trajectories by time not exceeding $\tmax$, and
$\mathcal{M} = \{ (c,w) \, | \, c = \pi, \, w \in i\R \}$ is the set of the Maxwell points.

We prove this using the Hadamard global diffeomorphism theorem~\cite{krantz-parks}, i.e., a smooth non-degenerate proper map of two connected and simply connected manifolds of the same dimensions is a diffeomorphism. Indeed, the manifolds $\mathcal{U}$ and $\interior{\A_{\tmax}}$ are three dimensional connected and simply connected.
Moreover, the smooth map $\Exp$ is non-degenerate on these domains, since according to Theorem~\ref{th-oblate-conj} and Proposition~\ref{prop-oblate-tmax} the inequality $\tmax(h) \leqslant \tconj(h)$ is satisfied. It remains to prove that this map is proper.

Let us prove that for any compact set $K \subset \interior{\A_{\tmax}}$ its pre-image $\Exp^{-1}{K}$ is compact as well.
Assume by contradiction that the set $\Exp^{-1}{K}$ is not compact. Then there exists a sequence of elements $(h^{(n)},t^{(n)})$ such that
(1) $h^{(n)}_1 \rightarrow -\infty$, or (2) $t^{(n)} \rightarrow 0$, or (3) $t^{(n)} \rightarrow \tmax(h^{(n)})$ when $n \rightarrow +\infty$.
But $K$ is compact, thus the sequence $\Exp(h^{(n)},t^{(n)})$ has a subsequence that converges to an element $q \in K$.
Passing to this subsequence we obtain that $q \in \partial\A$ in the case~(1), since $|h^{(n)}|$ is bounded, indeed,
$|h^{(n)}| < \frac{2\pi I_1}{t^{(n)}}$ and $t^{(n)}$ is separated from zero.
In the case~(2), we get $q = 0 \in \partial\A$.
This contradicts with the definition of the set $K$.

Consider now the case~(3). If there is a subsequence such that $\Kil{(h^{(n)})} < 0$, then $q \in \mathcal{M}$ and we get a contradiction.
Otherwise there exists a subsequence such that $\Kil(h^{(n)}) \geqslant 0$. In this case the corresponding sequence $\Exp{(h^{(n)},t^{(n)})}$ is unbounded,
but the set $K$ is compact. Hence, we get a contradiction.
\end{proof}

Fig.~\ref{fig-cut-front}--\ref{fig-cut-q3-0} presents extremal trajectories and the cut locus for $\bar{h}_3 \geqslant 0$
(there is the symmetry with respect to the plane $\Image{w} = 0$).
The thick solid line is the cut locus. The dashed curves are extremal trajectories that are optimal to the infinity.
The solid lines are extremal trajectories that lose their optimality at the time $t(h) = \frac{2\pi I_1}{|h|}$.
Fig.~\ref{fig-cut-front}~(left) corresponds to the factor by $\SO_{1,1}$-action.
Fig.~\ref{fig-cut-front}~(right) and Fig.~\ref{fig-cut-q3-0} shows the projection to the plane $\Image{w} = 0$.

\begin{figure}
\centering{
\minipage{0.5\textwidth}
  \centering{\includegraphics[width=\linewidth]{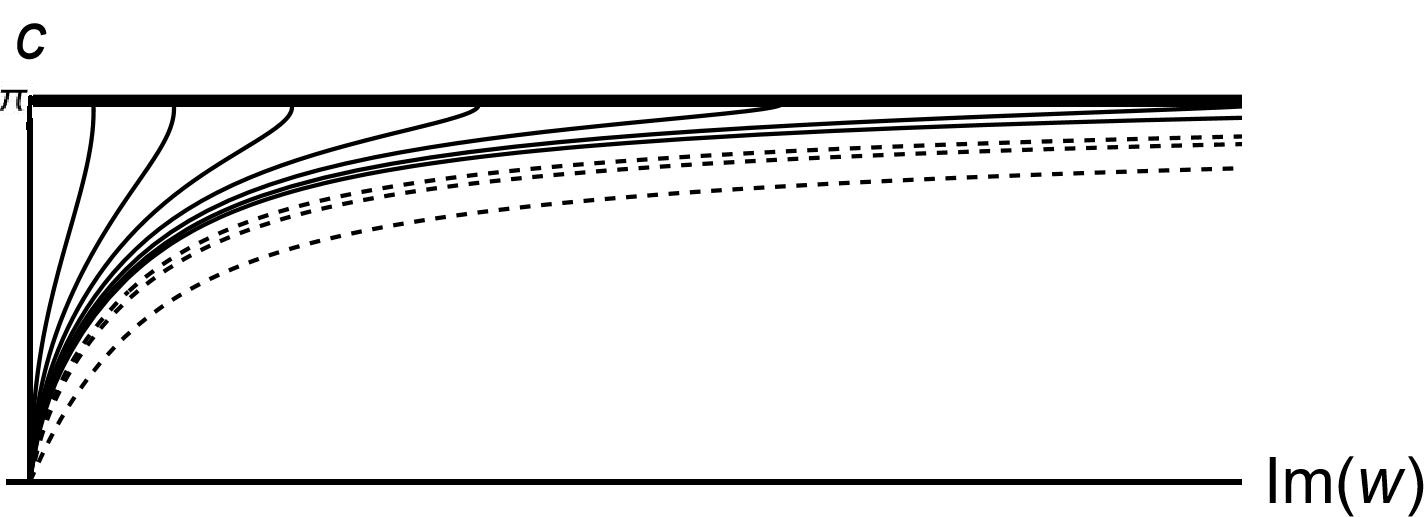} \\ $\SO_{1,1}$-factor}
\endminipage
\hfil
\minipage{0.25\textwidth}
  \centering{\includegraphics[width=\linewidth]{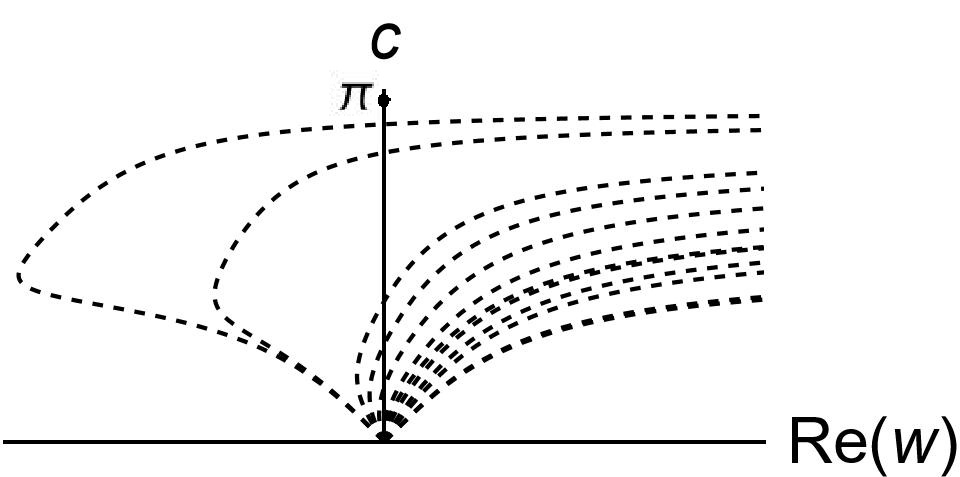} \\ The projection to the plane $\Image{w} = 0$}
\endminipage
}
\caption{\label{fig-cut-front}The extremal trajectories and the cut locus (the thick solid line) in the $\SO_{1,1}$-factor for $\bar{h}_3 \geqslant 0$ (left). The dashed lines are extremal trajectories that are optimal to the infinity. The projection of optimal to infinity extremal trajectories to the plane $\Image{w} = 0$ for $\bar{h}_3 = 3$ (right).}
\end{figure}

\begin{figure}
\centering{
\minipage{0.3\textwidth}
  \centering{\includegraphics[width=0.9\linewidth]{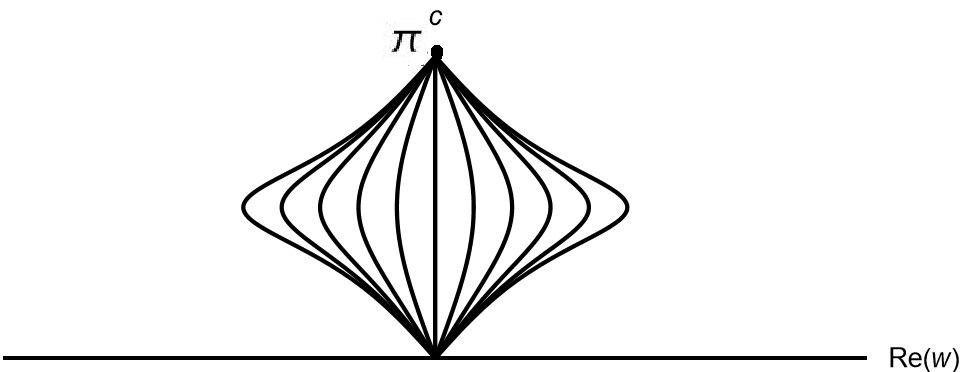} \\ $\bar{h}_3 = 0$}
\endminipage
\hfil
\minipage{0.3\textwidth}
  \centering{\includegraphics[width=0.9\linewidth]{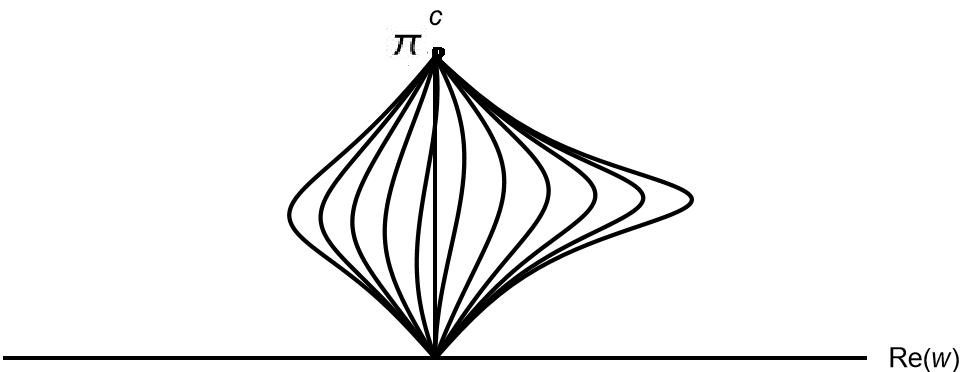} \\ $\bar{h}_3 = 0.2$}
\endminipage
\hfil
\minipage{0.3\textwidth}
  \centering{\includegraphics[width=0.9\linewidth]{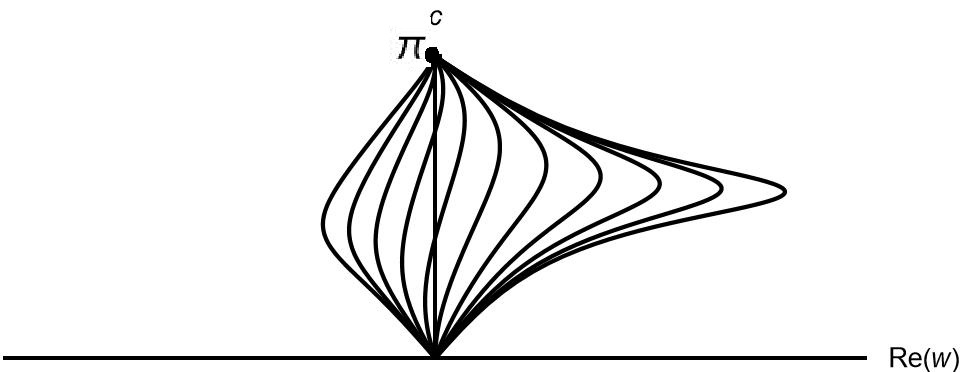} \\ $\bar{h}_3 = 0.4$}
\endminipage}
\caption{\label{fig-cut-q3-0}The extremal trajectories that lose their optimality for $\Kil{(h)} < 0$ and the cut locus in the projection to the plane $\Image{w} = 0$.}
\end{figure}

\begin{corollary}
\label{crl-oblate-exist}
If $\mu < 0$ and $q \in \A_{\tmax}$, then there exists the longest arc passing from the identity point $e$ to the point $q$.
\end{corollary}

\begin{proof}
Let us show that for any point of the set $\interior{\A_{\tmax}}$ the longest arc exists.
This follows from the fact that any point of this set can be reached by the unique extremal trajectory by Theorem~\ref{th-oblate-cut}.
The proof is based on the fields of extremals technic, see~\cite[Th.~17.2]{agrachev-sachkov}.
More precisely, consider the submanifold $L = \{e^{t\vec{H}}h \, | \, 0 < t < \tmax(h), \, h \in \g^*, H(h) = -\frac{1}{2} \} \subset T^*G$.
The dimension of this submanifold equals $\dim{G} = 3$ and the symplectic structure vanishes on it,
because it vanishes on the fiber of the cotangent bundle $T^*_{e}G = \g^*$ and the Hamiltonian vector field $\vec{H}$ is symplectically orthogonal to the level set of the Hamiltonian $H$. Hence, the submanifold $L$ is a Lagrange submanifold.
We proved (see the proof of Theorem~\ref{th-oblate-cut}) that the Lagrange manifold $L$ is diffeomorphic to $\interior{\A_{\tmax}}$.
Thus, for any point of the set $\interior{\A_{\tmax}}$ the unique extremal trajectory coming to this point is optimal.

The boundary of the set $\A_{\tmax}$ consists of two components $\partial\A_{\tmax} = \partial\A \cup \mathcal{M}$, i.e.,
the boundary of the whole attainable set and the Maxwell set $\mathcal{M}$.

Consider the first part of the boundary of the set $\A_{\tmax}$. This part is swept by the light-like extremal trajectories.
By Theorem~\ref{th-oblate-cut} the light-like extremal trajectories are optimal.

Finally, consider the second part of the boundary of the set $\A_{\tmax}$.
We prove that any extremal trajectory coming to a point $q = (\pi,w) \in \mathcal{M}$ is optimal, where $w \in i\R$.
Notice that there is a one-parametric family of such extremal trajectories.
By formulas~\eqref{eq-hyperbolic-time-geodesic} these extremal trajectories correspond to the following values of initial covectors
$$
\frac{h_1^2}{I_1} - \frac{h_2^2}{I_1} - \frac{h_3^2}{I_3} = 1, \qquad
\bar{h}_1^2 - \bar{h}_2^2 - \bar{h}_3^2 = 1, \qquad
\bar{h}_3 = -\frac{1}{\pi\mu}\arcsinh{w} = \const.
$$
It follows that for such extremal trajectories $|h| = \sqrt{\frac{I_1}{1 - \mu\bar{h}_3^2}} = \const$ and the corresponding Maxwell time
$t_w = \frac{2\pi I_1}{|h|} = \const$.

We use the semi-continuity of the Lorentzian distance, see~\cite[Lem.~4.4]{beem-ehrlich-easley}.
Namely, assume that $p,q \in M$ are two points of a Lorentzian manifold $M$ and there are two sequences of points $p_n \rightarrow p$, $q_n \rightarrow q$.
We will denote by $d(p,q)$ the Lorentzian distance from the point $p$ to the point $q$.
If $d(p,q) < +\infty$, then $d(p,q) \leqslant \liminf\limits_{n \rightarrow +\infty}{d(p_n,q_n)}$.
If $d(p,q) = +\infty$, then $d(p,q) \leqslant \lim\limits_{n \rightarrow +\infty}{d(p_n,q_n)} = +\infty$.

Let us take $p_n = p = (0,0)$ and $q_n \rightarrow q = (\pi,w)$, $q_n \in \interior{\A_{\tmax}}$.
It follows that $d(p,q) \leqslant \lim\limits_{n \rightarrow +\infty}{d(p,q_n)} = t_w$.
But since all extremal trajectories coming to the point $q$ have the Lorentzian length $t_w$, then $t_w \leqslant d(p,q)$.
Hence, $d(p,q) = t_w$ and any such extremal trajectory is optimal.
\end{proof}

\subsection{\label{sec-sub-Lorentzian-case}The sub-Lorentzian case}

In this section, we consider the sub-Lorentzian case that corresponds to the control set
$$
C_{-1} = \{u_1e_1 + u_2e_2 \in \g \, | \, I_1u_1^2 - I_1u_2 \geqslant 0, \ u_1 > 0\}.
$$
The corresponding optimal control problem reads as
$$
x(0) = e, \ x(t_1) = x_1, \qquad \dot{x}(t) = L_{x(t) *}u(t), \qquad \int\limits_0^{t_1}{\sqrt{I_1u_1^2(t) - I_2u_2^2(t)}\, dt} \rightarrow \max,
$$
where $x \in \Lip{([0,t_1],G)}$, $u \in L^{\infty}([0,t_1], C_{-1})$ and $x_1 \in G$, $t_1 > 0$ are fixed.
We obtain the sub-Lorentzian problem from Lorentzian one when $I_3 = +\infty$, or, equivalently, $\mu = -1$.

\begin{theorem}
\label{th-sub-Lorentzian}
\emph{(1)} The sub-Lorentzian normal extremal trajectories are products of two one-parameter subgroups as in Proposition~\emph{\ref{prop-geodesics}} with $\mu = -1$.\\
\emph{(2)} The sub-Lorentzian normal extremal trajectories have equations~\emph{\eqref{eq-hyperbolic-time-geodesic}}--\emph{\eqref{eq-hyperbolic-space-geodesic}} with $\mu = -1$.\\
\emph{(3)} The sub-Lorentzian abnormal extremal trajectories have at most one switch of light-like control.\\
\emph{(4)} \emph{\cite[Prop.~7\,(b)]{grong-vasiliev}} The sub-Lorentzian attainable set is\footnote{This formula coincides with the formula from paper~\cite{grong-vasiliev} up to the change $\Real{w} \ \leftrightarrow \ \Image{w}$.}
$$
\A^{sL} = \left\{ (c,w) \in G \, \Bigm| \, |\tan{c}| \geqslant \frac{\sqrt{(\Real{w})^2 + 2|\Image{w}|}}{1 - |\Image{w}|} \right\}.
$$
\emph{(5)} The sub-Lorentzian first conjugate time equals
$$
\tconj(h) = \left\{
\begin{array}{rcl}
\frac{2\pi I_1}{|h|}, & \text{if} & \Kil{(h)} < 0,\\
+\infty, & \text{if} & \Kil{(h)} \geqslant 0.\\
\end{array}
\right.
$$
\emph{(6)} The sub-Lorentzian first caustic is $\Conj = \{(\pi, w) \in G \, | \, w \in i\R \}$.\\
\emph{(7)} The sub-Lorentzian cut time equals the first conjugate time $\tcut(h) = \tconj(h)$.\\
\emph{(8)} The sub-Lorentzian cut locus is $\Cut = \{(\pi, w) \in G \, | \, w \in i\R \}$.\\
\emph{(9)} The sub-Lorentzian longest arcs exist for the final points in the set
$$
\A_{\tmax}^{sL}
= \left( \A^{sL} \cap \left\{ (c,w) \Bigm|
c < \pi - \arcsin{\frac{|\Real{w}|}{\sqrt{1+|w|^2}}}
\right\} \right) \cup \Cut.
$$
\end{theorem}

\begin{proof}
(1) The maximized Hamiltonian in the sub-Lorentzian case equals $H = \frac{1}{2}\left(-\frac{h_1^2}{I_1} + \frac{h_2^2}{I_2}\right)$.
Writing the Hamiltonian system $\dot{h}_i = \{H, h_i\}$, $i=1,2,3$ for normal extremal trajectories, we obtain
$$
\begin{array}{rcl}
\dot{h}_1(t) & = & -\frac{1}{I_1} h_2(t) h_3(t), \\
\dot{h}_2(t) & = & -\frac{1}{I_1} h_1(t) h_3(t), \\
\dot{h}_3(t) & = & 0, \\
\dot{x}(t) & = & L_{x(t)} \left( -\frac{h_1(t)}{I_1}e_1 + \frac{h_2(t)}{I_1}e_2 \right). \\
\end{array}
$$
This system coincides with~\eqref{eq-hamiltonian-system} when $\mu = -1$.
The same computations as in the proof of Proposition~\ref{prop-geodesics} give a representation of an extremal trajectory as a product of two one-parametric subgroups for $\mu = -1$.

(2) The previous argument implies the equations of normal extremal trajectories in coordinates.

(3) Let us study abnormal extremal trajectories. The corresponding Hamiltonian system reads as
$$
\begin{array}{rcl}
\dot{h}_1(t) & = & -\frac{1}{I_1} u_2(t) h_3(t), \\
\dot{h}_2(t) & = & -\frac{1}{I_1} u_1(t) h_3(t), \\
\dot{h}_3(t) & = & 0, \\
\dot{x}(t) & = & L_{x(t)} \left( u_1(t)e_1 + u_2(t)e_2 \right), \\
\end{array}
$$
where $H(h(t)) = 0$, $h_1(t) < 0$.
Moreover, the maximum condition implies that if $h_1^2(t) + h_2^2(t) \neq 0$, then the corresponding control $(u_1(t), u_2(t))$ is light-like, i.e.,
$u_1^2(t) + u_2^2(t) = 0$, $u_1(t) > 0$, and $\sgn{u_2(t)} = \sgn{h_2(t)}$.
If $h_1(t) = h_2(t) = 0$, then $h_3(t) \neq 0$ due to the non-triviality condition. Hence, $u_1(t) = u_2(t) = 0$.
Since the phase portrait of the abnormal Hamiltonian system looks like Fig.~\ref{fig-abnormal},
the abnormal extremal trajectories are concatenations of arcs of light-like extremal trajectories.
Moreover, there is at most one switching.

\begin{figure}
\centering{\includegraphics[width=0.4\linewidth]{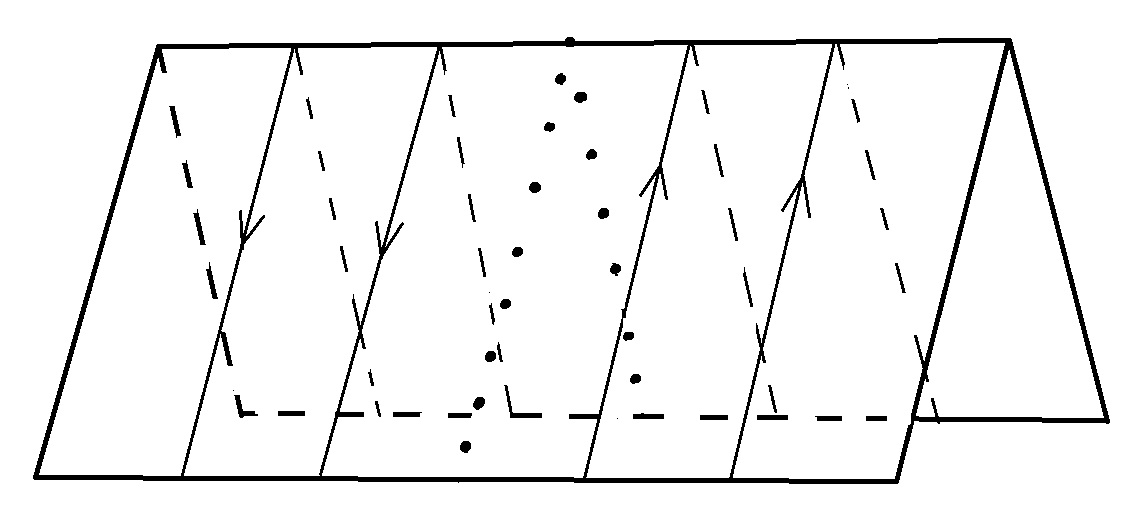}}
\caption{\label{fig-abnormal}The subsystem for covectors of the abnormal Hamiltonian system on the level surface $H(h) = 0$.}
\end{figure}

(4) Let us find the surface swept by the abnormal extremal trajectories. Since the abnormal extremal trajectories are concatenations of at most two arcs of light-like extremal trajectories,
the point of an abnormal extremal trajectory is a result of the multiplication~\eqref{eq-mult}, where for $i=1,2$ we have
$$
\tan{c_i} = |w_i|, \qquad e^{ic_1} = \cos{c_i} + i \sin{c_i}, \qquad
\cos{c_i} = \frac{1}{\sqrt{1+w_i^2}}, \qquad \sin{c_i} = \frac{|w_i|}{\sqrt{1+w_i^2}},
$$
where the signs of $w_1$ and $w_2$ are opposite.
From the second equation of the multiplication rule~\eqref{eq-mult} we get
$$
\begin{array}{rcccl}
\Real{w} & = & w_2\sqrt{1+w_1^2}\cos{c_1} + w_1\sqrt{1+w_2^2}\cos{c_2} & = & w_1 + w_2,\\
\Image{w} & = & w_2\sqrt{1+w_1^2}\sin{c_1} - w_1\sqrt{1+w_2^2}\sin{c_2} & = & \pm 2w_1 w_2 \quad \Rightarrow \quad |\Image{w}| = -2w_1w_2.\\
\end{array}
$$
It remains to calculate the first formula of the multiplication rule~\eqref{eq-mult} using the equality
$$
\tan{(\alpha + \beta + \gamma)} = \frac{\tan{\alpha} + \tan{\beta} + \tan{\gamma} - \tan{\alpha}\tan{\beta}\tan{\gamma}}{1 - \tan{\alpha}\tan{\beta} - \tan{\beta}\tan{\gamma} - \tan{\gamma}\tan{\alpha}}.
$$
We obtain
$$
\tan{c} = \frac{|w_1 - w_2|}{|1 + 2w_1w_2|} = \frac{\sqrt{(w_1-w_2)^2}} {|1+2w_1w_2|} = \frac{\sqrt{(w_1 + w_2)^2 - 4w_1w_2}}{|1 + 2w_1w_2|} =
\frac{\sqrt{(\Real{w})^2 + 2|\Image{w}|}}{|1 - |\Image{w}||}.
$$
We claim that the boundary of the attainable set is swept by the abnormal extremal trajectories. Indeed, the Pontryagin maximum principle implies that if a point lies on the boundary of the attainable set, then there exists an abnormal extremal trajectory coming to this point. Conversely, let us prove that any point $(c_0,w_0)$ of an abnormal extremal trajectory is a boundary point. Assume by contradiction that it is not a boundary point.
Note that since for the control set we have the following inclusion $C_{-1} \subset C_0$, the attainable set for the control set $C_{-1}$ is contained in the attainable set for the control set $C_0$, i.e., the well known attainable set for the Lorentzian symmetric case. This attainable set is bounded from below~\cite[Prop.~7\,(a)]{grong-vasiliev},
so our attainable set is bounded from below as well.
By our assumption the point $(c_0,w_0)$ lies inside the vertical section of the attainable set $\A_{w_0} = \{(c,w) \in \A \, | \, w = w_0\}$.
Since $\A$ is bounded from below, then $\inf{\A_{w_0}} > -\infty$. Hence, $(\inf{\A_{w_0}}, w_0) \in \partial\A$.
But there is no any abnormal extremal trajectory coming to this boundary point. We get a contradiction.

(5) The computations of the first conjugate time in the proof of Theorem~\ref{th-oblate-conj} are still valid, including the inequalities~\eqref{eq-t-conj-coeff0}, \eqref{eq-s-conj-coeff0}, but instead of inequalities~\eqref{eq-t-conj-coeff1}, \eqref{eq-s-conj-coeff1} we obtain equalities, but this does not affect to the first conjugate time.

(6) The previous argument immediately implies that the sub-Lorentzian first caustic coincides with the Lorentzian first caustic in the oblate case.

(7)--(9) The proof is the same as the proof of Theorem~\ref{th-oblate-cut} and Corollary~\ref{crl-oblate-exist}.
\end{proof}

\begin{remark}
\label{rem-non-convergence}
For the axisymmetric one-parameter series of Riemannian problems on the Lie group $\SU_{1,1}$ studied in paper~\cite{podobryaev-sachkov} there is a convergence to the sub-Riemannian one studied in papers~\cite{berestovskii-zubareva,boscain-rossi}. More precisely, the Riemannian cut time and the cut locus depending on the parameter converge to the sub-Riemannian cut time and the locus~\cite[Th.~5]{podobryaev-sachkov}.

From Theorem~\ref{th-sub-Lorentzian} we see that for the one-parameter series of Lorentzian problems such that the corresponding control set and the cost functional estimate the sub-Lorentzian ones the situation is different.
Despite of the convergence of the Lorentzian normal extremal trajectories to the sub-Lorentzian ones when $\mu \rightarrow -1$,
the Lorentzian attainable sets converge to the set defined by the inequality
$$
|\tan{c}| \geqslant \frac{|\Real{w}|}{\sqrt{1+(\Image{w})^2}}
$$
which contains the sub-Lorentzian attainable set, since
$$
\frac{\sqrt{(\Real{w})^2 + 2|\Image{w}|}}{|1 - |\Image{w}||} \geqslant \frac{|\Real{w}|}{\sqrt{1+(\Image{w})^2}}
$$
because
$$
\left((\Real{w})^2 + 2|\Image{w}|\right) \left(1+(\Image{w})^2\right) \geqslant (\Real{w})^2 \left(1 - |\Image{w}|\right)^2.
$$
Moreover, the first conjugate time, the first caustic, the cut time and the cut locus do not depend on the parameter $\mu$.
\end{remark}

\section{\label{sec-prolate}The prolate case}

Consider now the prolate case $\mu > 0$.

\subsection{\label{sec-prolate-atset}The attainable set}

\begin{theorem}
\label{th-prolate-atset}
If $\mu > 0$, then the attainable set coincides with the whole group $\A = G$.
\end{theorem}

To prove this theorem we divide it on several steps given by the following lemmas based on the fact that in the prolate case the control set contains the symmetric control set $C_0 \subset C$. We construct a sequence of admissible controls first to get a point $(c,w)$ with $c < 0$ and $|w|$ big enough, second to get a point on the $c$-axis,
third to get a given point.

\begin{lemma}
\label{lem-up}
Let $\{(c_0,w) \, | \, |w| = r_0 \} \subset \A$, then
$\{(c,w) \, | \, |w| = r_0 \} \subset \A$ for any $c \geqslant c_0$
\end{lemma}

\begin{proof}
Consider the admissible trajectory starting from the point $e = (0,0)$ with the constant control $ u = (1,0,0)$.
Since $\Kil{(u)} < 0$, by formulas~\eqref{eq-group-exp} we obtain $c_u(t) = \arctan{\left( \tan{\frac{t}{2}} \right)} = \frac{t}{2}$, $w_u(t) = 0$.
Then by the multiplication rule~\eqref{eq-mult} we get
$$
(c_0,w) \cdot (c_u(t), w_u(t)) = (c_0 + c_u(t), w e^{-ic_u(t)}) = \left(c_0 + \frac{t}{2}, we^{-i\frac{t}{2}}\right).
$$
Hence, to attain the point $(c,w)$ we need to start from the point $(c_0, we^{i(c-c_0)}) \in \A$.
\end{proof}

\begin{lemma}
\label{lem-pi-2}
\emph{(1)} The inclusion $\left\{\left(\frac{\pi}{2},w\right) \, | \, w \in \C \right\} \subset \A$ is satisfied.\\
\emph{(2)} Let $(c_0,0) \in \A$, then $\{ (c,w) \, | \, c \geqslant c_0 + \frac{\pi}{2}, \ w \in \C\} \subset \A$.
\end{lemma}

\begin{proof}
(1) Consider a family of admissible curves $(c_u(t), w_u(t))$ starting from the point $e = (0,0)$ with a constant control $u = (1,u_2,u_3)$,
such that $1 - u_2^2 - u_3^2 = 0$. Since $\Kil{(u)} = 0$, by formulas~\eqref{eq-group-exp} we obtain
$$
c_u(t) = \arctan{\frac{t}{2}}, \qquad w_u(t) = \frac{t}{2}(u_2 + iu_3), \qquad \sqrt{u_1^2+u_2^2} = 1.
$$
Since the function $\arctan$ increases up to $\frac{\pi}{2}$ the statement~(1) follows from Lemma~\ref{lem-up}.

(2) Repeat the proof of the statement~(1) starting from the point $(c_0,0)$. Due to the multiplication law~\eqref{eq-mult} we have
$$
(c_0,0) \cdot (c_u(t),w_u(t)) = (c_0 + c_u(t), w_u(t)e^{ic_0}).
$$
It remains to apply Lemma~\ref{lem-up}.
\end{proof}

Also we need the following technical lemma.

\begin{lemma}
\label{lem-function}
Consider the function $f : [0,+\infty) \rightarrow \R$ defined by the formula $f(s) = \frac{s}{\sqrt{1+s^2}}$. Then\\
\emph{(1)} The function $f$ increases.\\
\emph{(2)} $\lim\limits_{s \rightarrow +\infty}{f(s)} = 1$.\\
\emph{(3)} $\lim\limits_{s \rightarrow 0+}{\frac{f(s)}{f(as)}} = \frac{1}{a}$ for $a \neq 0$.
\end{lemma}

\begin{proof}
(1) It is easy to check that $f'(s) = \frac{1}{(1+s^2)^{3/2}} > 0$.
(2) This follows from $f(s) = \frac{1}{\sqrt{1 + 1/s^2}} \rightarrow 1$ when $s \rightarrow +\infty$.
(3) Note that $\frac{f(s)}{f(as)} = \frac{\sqrt{1+a^2s^2}}{a\sqrt{1+s^2}} \rightarrow \frac{1}{a}$ when $s \rightarrow 0+$.
\end{proof}

Consider an admissible curve $\gamma_u(\cdot)$ with a constant light-like control $u$ such that $\Kil{(u)} \neq 0$.
For such a control we have
$$
\begin{array}{l}
I_1 \bar{u}_1^2 - I_1 \bar{u}_2^2 - I_3 \bar{u}_3^2 = 0,\\
\bar{u}_1^2 - \bar{u}_2^2 - \bar{u}_3^2 = -1,\\
\end{array}
\quad \Rightarrow \quad
\begin{array}{l}
(\mu+1)(\bar{u}_1^2 - \bar{u}_2^2) - \bar{u}_3^2 = 0,\\
\mu(\bar{u}_1^2 - \bar{u}_2^2) - 1 = 0,\\
\end{array}
\quad \Rightarrow \quad
\begin{array}{l}
\bar{u}_1^2 - \bar{u}_2^2  = \frac{1}{\mu},\\
\bar{u}_3^2 = \frac{\mu+1}{\mu}.\\
\end{array}
$$
Since $\Kil{(u)} > 0$, using~\eqref{eq-group-exp} for $\gamma_u(t) = (c_u(t), w_u(t))$ we obtain
$$
c_u(t) = \arctan{\left(\bar{u}_1\tanh{\frac{t}{2}}\right)}, \qquad w_u(t) = \sinh{\frac{t}{2}} \left(\bar{u}_2 + i\bar{u}_3\right).
$$
Moreover,
$$
|w_u(t)| = \sinh{\left(\frac{t}{2}\right)} \sqrt{\bar{u}_2^2 + \bar{u}_3^2} = \sinh{\left(\frac{t}{2}\right)} \sqrt{\bar{u}_1^2 + 1}.
$$
In other words,
\begin{equation}
\label{eq-c-w}
c_u(t) = \arctan{(\bar{u}_1 f(s))}, \qquad |w_u(t)| = \sqrt{\bar{u}_1^2 + 1} \cdot s, \qquad \text{where} \qquad s = \sinh{\frac{t}{2}}.
\end{equation}

\begin{lemma}
\label{lem-time}
There exists a time $t_1 > 0$ such that
$$
\frac{\bar{u}_1 f\left(\sinh{\frac{t_1}{2}}\right)}{f(|w_u(t_1)|)} < 1.
$$
\end{lemma}

\begin{proof}
Due to formulas~\eqref{eq-c-w} and Lemma~\ref{lem-function}\,(3) we have
$$
\frac{\bar{u}_1 f\left(\sinh{\frac{t}{2}}\right)}{f(|w_u(t)|)} = \frac{\bar{u}_1 f(s)}{f\left( \sqrt{\bar{u}_1^2 + 1}\cdot s \right)} \rightarrow
\frac{\bar{u}_1}{\sqrt{\bar{u}_1^2 + 1}} < 1 \qquad \text{when} \qquad t \rightarrow 0+.
$$
Hence, there exists such $t_1 > 0$ that the inequality is satisfied.
\end{proof}

\begin{lemma}
\label{lem-values}
Assume that a control $u$ is such that $\bar{u} = \left( \bar{u}_1, \bar{u}_2, \pm \sqrt{\frac{\mu+1}{\mu}} \right)$, where $\bar{u}_1^2 - \bar{u}_2^2 = \frac{1}{\mu}$.
Let $t_1 > 0$ be such that the statement of Lemma~\emph{\ref{lem-time}} is satisfied.
Then $\varphi_u(t_1) = \arg{\{ w_u(t_1)e^{ic_u(t_1)} \}}$ takes any values except $2\pi\Z$ and $\pi + 2\pi\Z$.
\end{lemma}

\begin{proof}
It is clear that the function $\varphi_u(t_1)$ as a function of variable $u$ is continuous.
So, it is sufficient to prove that if $\bar{u}_3 = \sqrt{\frac{\mu+1}{\mu}}$, then
$$
\lim\limits_{\bar{u}_2 \rightarrow +\infty}{\varphi_u(t_1)} = 0, \qquad \lim\limits_{\bar{u}_2 \rightarrow -\infty}{\varphi_u(t_1)} = \pi.
$$
Indeed, $\varphi_u(t_1) = \arg{w_u(t_1)} + c_u(t_1) = \arg{u} + c_u(t_1)$ and
$\arg{u} \rightarrow \pm \frac{\pi}{2}$ and $c_u(t_1) \rightarrow \frac{\pi}{2}$ while $\bar{u}_2 \rightarrow \pm \infty$.
\end{proof}

Let us show that we can go down with respect to the variable $c$.

\begin{lemma}
\label{lem-down}
\emph{(1)} For any $c_0 \in \R$ and for almost all $e^{i\varphi} \in S^1$ there exist $r_0 > 0$, a control $u$
\emph{(}where $\bar{u} = \left( \bar{u}_1, \bar{u}_2, \pm \sqrt{\frac{\mu+1}{\mu}} \right)$ and $\bar{u}_1^2 - \bar{u}_2^2 = \frac{1}{\mu}$\emph{)} and $t_1 > 0$ such that
for any $w_0 = |w_0|e^{i\varphi}$, where $|w_0| > r_0$ for
$(\hat{c}_u(t_1), \hat{w}_u(t_1)) = (c_0, w_0) \cdot (c_u(t_1), w_u(t_1))$ the inequality $\hat{c}_u(t_1) < c_0$ is satisfied.\\
\emph{(2)} The inequality
$$
c_0 - \hat{c}_u(t_1) > \arctan{f(|w_u(t_1)|)} - c_u(t_1)
$$
is satisfied.\\
\emph{(3)} Moreover, choosing suitable $|w_0| > r_0$ we can get any value
$$
|\hat{w}_u(t_1)| \geqslant \sqrt{|w_u(t_1)|^2(1+r_0^2) + r_0^2(1+|w_u(t_1)|^2)}.
$$
\end{lemma}

\begin{proof}
(1) By Lemma~\ref{lem-values} for almost all $e^{i\varphi} \in S^1$ there exist $t_1 > 0$ and $u$
(where $\bar{u} = \left( \bar{u}_1, \bar{u}_2, \pm \sqrt{\frac{\mu+1}{\mu}} \right)$, $\bar{u}_1^2 - \bar{u}_2^2 = \frac{1}{\mu}$) such that
\begin{equation}
\label{eq-pi-2}
\arg{\{ w_u(t_1)e^{ic_u(t_1)} \}} = \arg{\{ e^{i\varphi}e^{-ic_0} \}} + \frac{\pi}{2},
\end{equation}
and $\frac{\bar{u}_1 f\left( \sinh{\frac{t}{2}} \right)}{f(|w_u(t_1)|)} < 1$.

Due to Lemma~\ref{lem-function}\,(1--2) there exists $r_0 > 0$ such that
$$
\frac{\bar{u}_1 f\left( \sinh{\frac{t}{2}} \right)}{f(|w_u(t_1)|)} < f(r_0) < 1,
$$
and for any $w_0 = |w_0|e^{i\varphi}$, $|w_0| > r_0$ we have $f(r_0) < f(|w_0|) < 1$.

Hence, $\frac{\bar{u}_1 f\left( \sinh{\frac{t}{2}} \right)}{f(|w_u(t_1)|)} < f(|w_0|)$, i.e., by formula~\eqref{eq-c-w} and since the function $\arctan$ increases we have
\begin{equation}
\label{eq-arctan}
c_u(t_1) = \arctan{\left( \bar{u}_1f\left( \sinh{\frac{t_1}{2}} \right)\right)} < \arctan{\left( f(|w_u(t_1)|)f(|w_0|) \right)}.
\end{equation}
From the multiplication rule~\eqref{eq-mult} we obtain
$$
\hat{c}_u(t_1) = c_0 + c_u(t_1) - \arctan{\left( f(|w_u(t_1)|)f(|w_0|) \right)},
$$
because formula~\eqref{eq-pi-2} implies that
$$
\arg{\{ w_0\bar{w}_u(t_1)e^{-i(c_0 + c_u(t_1))} \}} = -\frac{\pi}{2}.
$$
Finally, it follows from inequality~\eqref{eq-arctan} that $\hat{c}_u(t_1) < c_0$.

(2) It follows from the fact that the value of
$$
c_0 - \hat{c}_u(t_1) = \arctan{\left( f(|w_u(t_1)|)f(|w_0|) \right)} - c_u(t_1)
$$
increases when $|w_0|$ is increasing.

(3) From the multiplication rule~\eqref{eq-mult} we have
$$
\hat{w}_u(t_1) = w_u(t_1)e^{ic_0}\sqrt{1 + |w_0|^2} + w_0e^{-ic_u(t_1)}\sqrt{1 + |w_u(t_1)|^2}.
$$
The equation~\eqref{eq-pi-2} means that these two terms are perpendicular to each other. Hence, by the Pythagoras theorem
$$
|\hat{w}_u(t_1)|^2 = |w_u(t_1)|^2 (1 + |w_0|^2) + |w_0|^2 (1 + |w_u(t_1)|^2).
$$
This expression increases as a function of the variable $|w_0|$. This implies the statement~(3).
\end{proof}

\begin{lemma}
\label{lem-any-c}
For any $c \in \R$ there exists $w \in \C$ \emph{(}with sufficiently big $|w|$\emph{)} such that $(c,w) \in \A$.
\end{lemma}

\begin{proof}
We know from Lemma~\ref{lem-pi-2}\,(1) that $\left\{ \left(\frac{\pi}{2}, w \right) \, | \, w \in \C \right\} \subset \A$.
Due to Lemma~\ref{lem-pi-2}\,(2) we may assume that $c < \frac{\pi}{2}$.
Starting from this set and repeatedly applying Lemma~\ref{lem-down} we get that for some $c_1 < c$ there exists $w$ such that $(c_1,w) \in \A$.

Indeed, at every step we have $(c,w) \in \A$ for sufficiently big $|w|$ and almost all $\arg{w}$.
Moreover, at any step the value $c$ decreases and due to Lemma~\ref{lem-down}\,(2) the next step of this decreasing can be made not less than the previous one.

It remains to apply the constant control $(1,0,0)$ to reach the required $c$ from the point $(c_1,w)$.
Note that in this case $\arg{w}$ will change, but $|w|$ will stay the same according to the multiplication law~\eqref{eq-mult}.
\end{proof}

\begin{lemma}
\label{lem-axis}
The inclusion $\{ (c,0) \, | \, c \in \R \} \subset \A$ is satisfied.
\end{lemma}

\begin{proof}
Let $c \in \R$ be arbitrary. There exists a sufficiently big $|w_0|$ such that $\left( c - \frac{\pi}{2}, w_0 \right) \in \A$ by Lemma~\ref{lem-any-c}.
Let us consider a curve starting from the point $e = (0,0)$ with a constant control $u = (1,u_2,u_3)$ such that $u_2^2 + u_3^2 = 1$:
$$
c_u(t) = \arctan{\frac{t}{2}}, \qquad w_u(t) = \frac{t}{2} (u_2 + iu_3).
$$
Then by the multiplication rule~\eqref{eq-mult} we can calculate $(\hat{c}_u(t), \hat{w}_u(t)) = \left(c - \frac{\pi}{2}, w_0 \right) \cdot (c_u(t), w_u(t))$.
We claim that there exist $u$ and $t$ such that $\hat{w}_u(t) = 0$.
Indeed, we need
$$
w_u(t) \sqrt{1 + |w_0|^2} e^{ic_0} = -w_0 \sqrt{1 + |w_u(t)|^2} e^{-ic_u(t)}.
$$
The modules of these complex numbers are equal if and only if $f(|w_u(t)|) = f(|w_0|)$ and we can choose $t$ such that this equality will be satisfied.
Next, we can choose $u$ such that the arguments of these two complex numbers will be equal.
We obtain
$$
c - \frac{\pi}{2} < \hat{c}_u(t) < c - \frac{\pi}{2} + \frac{\pi}{2} = c,
$$
since $\Image{\left( w_0 \bar{w}_u e^{-i(c_0 + c_u(t))} \right)} = 0$ in this case.
It remains to start from the point $(\hat{c}_u(t),0) \in \A$ with the constant control $(1,0,0)$ to reach the point $(c,0)$.
\end{proof}

\begin{proof}[Proof of Theorem~\emph{\ref{th-prolate-atset}}]
Let us show that $(c,w) \in \A$ for an arbitrary point $(c,w) \in G$.
By Lemma~\ref{lem-axis} the point $\left( c - \frac{\pi}{2}, 0 \right)$ is attainable.
It follows from Lemma~\ref{lem-pi-2}\,(2) that $(c,w) \in \A$.
\end{proof}

\begin{corollary}
\label{crl-prolate-existence}
For any point $x \in G$ the supremum of the length of all admissible curves from the point $e$ to the point $x$ is equal to $+\infty$.
In particular, there is no the longest arc from the starting point $e$ to the point $x$.
\end{corollary}

\begin{proof}
It is sufficient to prove that for any $x \in G$ there is an admissible loop passing through the point $x$. Indeed, in this case including this loop in an admissible curve repeatedly we get a sequence of admissible curves with arbitrary big length. Due to the left-invariance it is enough to find an admissible loop through the point $e$.
But the point $(-1,0)$ is attainable by Theorem~\ref{th-prolate-atset}, it remains to note that the segment from the point $(-1,0)$ to the point $e = (0,0)$ is an admissible curve.
\end{proof}

\begin{remark}
\label{rem-prolate-cut}
Since there are no optimal solutions in the prolate case there is no sense in the notions of the cut time and the cut locus.
However, the first conjugate time, the first Maxwell time for symmetries and the corresponding Maxwell set are well defined.
We study these objects in the next two sections.
\end{remark}

\subsection{\label{sec-prolate-conj}The conjugate points}

\begin{theorem}
\label{th-prolate-conj}
If $\mu > 0$, then
$$
\tconj(h) = \left\{
\begin{array}{lcl}
\frac{2 \tau_{\mathrm{conj}}(\bar{h}_3) I_1}{|h|}, & \text{if} & \Kil{(h)} < 0,\\
+\infty & \text{if} & \Kil{(h)} = 0,\\
\end{array}
\right.
$$
where $\tau_{\mathrm{conj}}$ is the first positive root of the equation $\tan{\tau} = -\tau\mu\frac{-1-\bar{h}_3^2}{-1+\mu\bar{h}_3^2}$.
Moreover, $\tau_{\mathrm{conj}}(\bar{h}_3) \in [\frac{\pi}{2}, \pi)$.
\end{theorem}

\begin{proof}
Note that in the prolate case $\mu > 0$ almost all extremal trajectories have the form~\eqref{eq-hyperbolic-time-geodesic}.
Moreover, for the time-like extremal trajectories the corresponding $\bar{h}_3 \in (-\frac{1}{\sqrt{\mu}}, \frac{1}{\sqrt{\mu}})$ by Lemma~\ref{lem-h3}\,(2).

Repeating the proof of Theorem~\ref{th-oblate-conj} we obtain that the first conjugate time is determined by the smallest root of the equation $\sin{\tau} = 0$ and the equation~\eqref{eq-tgth} which takes the form
\begin{equation}
\label{eq-prolate-conj}
\tan{\tau} = -\tau\mu\frac{-1-\bar{h}_3^2}{-1+\mu\bar{h}_3^2},
\end{equation}
where $\mu > 0$, $-1-\bar{h}_3^2 < 0$ and $-1 + \mu\bar{h}_3^2 < 0$ for $\bar{h}_3^2 < \frac{1}{\mu}$.
This implies that equation~\eqref{eq-prolate-conj} has a zero at the interval $(\frac{\pi}{2}, \pi)$, see Fig.~\ref{fig-prolate-conj}.
Of course, this zero is less than the first positive zero of the equation $\sin{\tau} = 0$, i.e., $\pi$.

For light-like extremal trajectories such that $\Kil{(h)} \neq 0$ we have $\bar{h}_3 = \pm \frac{1}{\sqrt{\mu}}$.
The conjugate time is determined by the first positive root of the expression in the square brackets in formula~\eqref{eq-conj-full}.
This root equals $\frac{\pi}{2}$.

\begin{figure}
\centering{
     \begin{minipage}[h]{0.45\linewidth}
        \center{\includegraphics[width=0.7\linewidth]{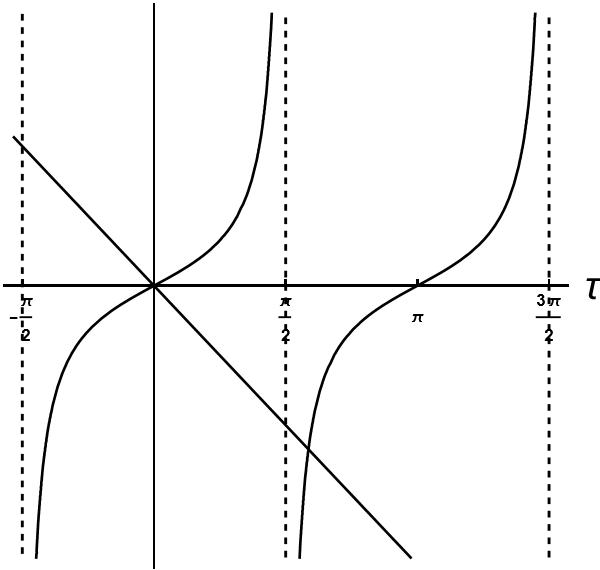}}
        \caption{\label{fig-prolate-conj}The first positive root of equation~\eqref{eq-prolate-conj} determines the first conjugate time.}
     \end{minipage}
     \hfill
     \begin{minipage}[h]{0.45\linewidth}
        \center{\includegraphics[width=0.7\linewidth]{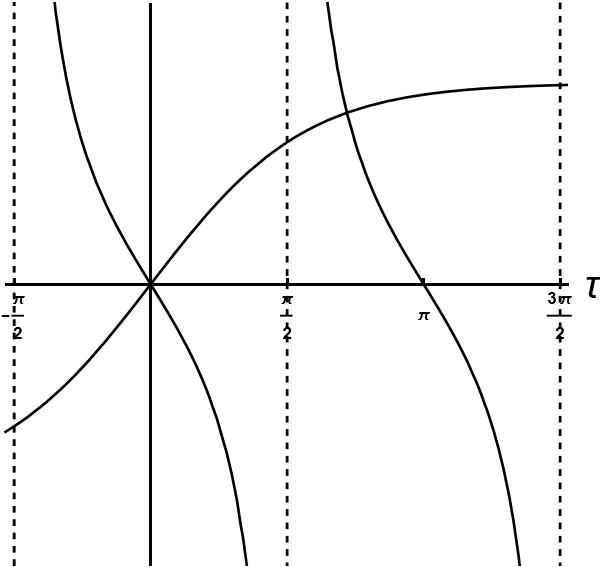}}
        \caption{\label{fig-prolate-maxwell}The first Maxwell time for symmetries in the prolate case.}
     \end{minipage}
     }
\end{figure}

Remind that $t = \frac{2\tau I_1}{|h|}$.
It remains to prove that there are no conjugate points on the light-like extremal trajectories with initial covectors $h$ such that $\Kil{(h)} = 0$.
Since the first positive root of equation~\eqref{eq-prolate-conj} is located at the interval $[\frac{\pi}{2}, \pi)$, then there are no conjugate points on the time interval $[0,t_1]$ for time-like extremal trajectories with initial covectors $h$ such that $|h| < \frac{\pi I_1}{t_1}$.
Hence, we can repeat the argumentation from the proof of Theorem~\ref{th-oblate-conj} using the homotopy invariance of the number of conjugate points and
choosing a curve of initial covectors for time-like extremal trajectories $\gamma : [0,1] \rightarrow \g^*$ such that $|\gamma(s)| < \frac{\pi I_1}{t_1}$.
\end{proof}

\subsection{\label{sec-prolate-maxwell}The Maxwell set}

The symmetry group for the exponential map contains the group $\mathrm{Sym} = \mathrm{O}_{1,1} \times \Z_2$ as in the oblate case, see Proposition~\ref{prop-sym}.

\begin{theorem}
\label{th-prolate-maxwell}
\emph{(1)} The first Maxwell time for symmetries equals
$$
\tmax = \left\{
\begin{array}{lcr}
\frac{2\tau_3(\bar{h}_3)I_1}{|h|}, & \text{if} & \Kil{(h)} < 0,\\
+\infty, & \text{if} & \Kil{(h)} = 0,\\
\end{array}
\right.
$$
where $\tau_3(\bar{h}_3)$ is the first positive root of the equation $q_3(\tau) = 0$, see~\emph{\eqref{eq-hyperbolic-time-geodesic}}.
Moreover, $\tau_3(\bar{h}_3) \in (\frac{\pi}{2}, \pi)$.\\
\emph{(2)} The first conjugate time is less or equal to the first Maxwell time for symmetries
$$
\tconj(h) \leqslant \tmax(h),
$$
where there is the equality only if $\Kil{(h)} = 0$, in this case $\tconj(h) = \tmax(h) = +\infty$.
\end{theorem}

\begin{figure}
\centering{
     \begin{minipage}[h]{0.45\linewidth}
        \center{\includegraphics[width=0.7\linewidth]{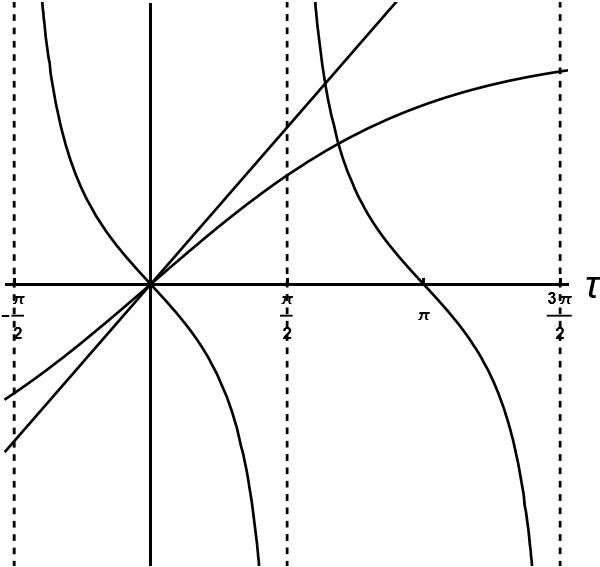}}
        \caption{\label{fig-prolate-maxwell-conj}The first conjugate time is less than the first Maxwell time for symmetries in the prolate case.}
     \end{minipage}
     \hfill
     \begin{minipage}[h]{0.45\linewidth}
        \center{\includegraphics[width=\linewidth]{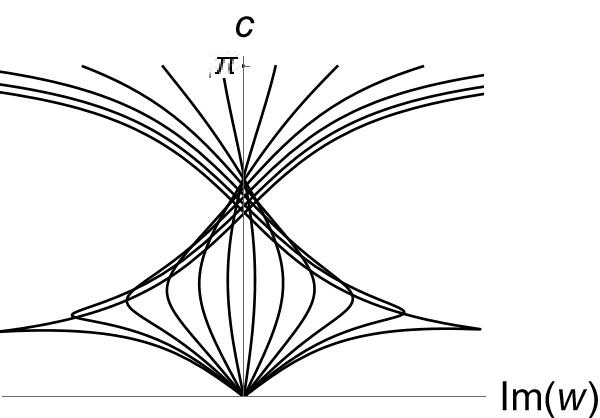}}
        \caption{\label{fig-prolate-maxwell-stratum}The extremal trajectories in the prolate case ($\SO_{1,1}$-factor).}
     \end{minipage}
     }
\end{figure}

\begin{proof}
Note that if $\Kil{(h)} = 0$, i.e., $\bar{h}_3 = 0$, then the first conjugate time equals $+\infty$ and the corresponding extremal trajectories lie in the plane $\Image{w} = 0$.
Assume that $\Kil{(h)} < 0$ and $\bar{h}_3 > 0$ (the case $\bar{h}_3 < 0$ is similar).

(1) We should compare the first Maxwell times for the hyperbolic rotations and the symmetries with respect to the planes $h_3 = 0$ and $h_2 = 0$.
Similarly to the proof of Proposition~\ref{prop-oblate-tmax} we need to compare the number $\pi$ and the first positive root of the first equation of~\eqref{eq-oblate-maxwell}, i.e.,
$$
\frac{1}{\bar{h}_3} \tanh{(\tau\mu\bar{h}_3)} = -\tan{\tau}.
$$
We see that the graphs of the functions at the left-hand and the right-hand sides of this equation look like graphs on Fig.~\ref{fig-prolate-maxwell}.
Thus, the first positive root is located at the interval $(\frac{\pi}{2}, \pi)$ and automatically is less than $\pi$.

(2) Since the first conjugate time for extremal trajectories with initial covectors $h$ such that $\Kil{(h)} \neq 0$ is equal to $+\infty$,
we may assume that $h$ is such that $\Kil{(h)} < 0$.
The first conjugate time for time-like extremal trajectories (i.e., $\bar{h}_3 \in (-\frac{1}{\sqrt{\mu}},\frac{1}{\sqrt{\mu}})$) determines by equation~\eqref{eq-prolate-conj} which is equivalent to
$$
-\tan{\tau} = \tau\mu\frac{-1-\bar{h}_3^2}{-1+\mu\bar{h}_3^2}.
$$
Note that the coefficient of the line at the left-hand side of this equation is positive.
Moreover, this coefficient is greater than $\frac{d}{d\tau}|_{\tau = 0} \left( \frac{1}{\bar{h}_3} \tanh{(\tau\mu\bar{h}_3)} \right) = \mu$.
Indeed, since $-1+\mu\bar{h}_3^2 < 0$ and remember that $\bar{h}_3 \in (-\frac{1}{\sqrt{\mu}}, \frac{1}{\sqrt{\mu}})$ we have
$$
\mu\frac{-1-\bar{h}_3^2}{-1+\mu\bar{h}_3^2} > \mu \qquad \Leftrightarrow \qquad -1-\bar{h}_3^2 < -1+\mu\bar{h}_3^2.
$$
Hence, this line intersects the graph of the function $-\tan{\tau}$ earlier than the graph of $\frac{1}{\bar{h}_3} \tanh{(\tau\mu\bar{h}_3)}$ intersects the graph of $-\tan{\tau}$, see Fig.~\ref{fig-prolate-maxwell-conj}.

For $\bar{h}_3 = \pm \frac{1}{\sqrt{\mu}}$ the first conjugate time equals $\frac{\pi}{2}$ which is less than the first Maxwell time as well.
\end{proof}

\begin{remark}
\label{rem-prolate-maxwell-stratum}
In the prolate case, conjugate points appears earlier than Maxwell points unlike the oblate case and the situation in (sub-)Riemannian geometry.
Fig.~\ref{fig-prolate-maxwell-stratum} presents the extremal trajectories in the prolate case, the picture is plotted in the factor of the action of the group $\SO_{1,1}$.
There exists an envelope curve for extremal trajectories (the caustic).
Then extremal trajectories with opposite values of $\bar{h}_3$ meet one another at the plane $\Image{w} = 0$.
This is the first Maxwell stratum corresponding the symmetry with respect to the plane $h_3 = 0$ in the pre-image of the exponential map
(with respect to the plane $\Image{w} = 0$ in the image of the exponential map).
\end{remark}

\end{document}